\documentclass[12pt]{article}
\usepackage[top=1in,bottom=1in,left=1in,right=1in]{geometry}
\usepackage[T1]{fontenc}
\usepackage{lmodern}
\usepackage{microtype}
\usepackage{indentfirst}
\usepackage{amsfonts,amsmath,amsthm,amssymb}
\usepackage{mathrsfs}
\usepackage{enumitem}
\usepackage{xcolor}
\usepackage{array}
\usepackage{cite}
\usepackage{hyperref}

\hypersetup{
	colorlinks=true,
	linkcolor=red,
	filecolor=blue,
	urlcolor=red,
	citecolor=blue,
	pdftitle={The Fermat-Type Matrix Equation over GL2(Z): Solvability and Integral-Conjugacy Orbits},
	pdfauthor={Hongjian Li and Weilin Zhang},
	pdfsubject={The Fermat-type matrix equation over GL2(Z)},
	pdfkeywords={Fermat-type matrix equation, GL2Z, integral conjugacy, orbit decomposition, torsion matrices, Cayley-Hamilton theorem}
}

\allowdisplaybreaks
\numberwithin{equation}{section}

\newtheorem{theorem}{Theorem}[section]

\newtheorem{lemma}{Lemma}[section]
\newtheorem{proposition}{Proposition}[section]
\newtheorem{corollary}{Corollary}[section]
\newtheorem{remark}{Remark}[section]

\def\G1{G^\mathcal{C}}

\newcommand{\PSL}{{\rm PSL}}

\newcommand{\SL}{{\rm SL}}
\newcommand{\GL}{{\rm GL}}

\newcommand{\tr}{\operatorname{tr}}
\newcommand{\ord}{\operatorname{ord}}
\newcommand{\Id}{I_2}
\newcommand{\smat}[1]{\left(\begin{smallmatrix}#1\end{smallmatrix}\right)}

\begin{document}
\title{The Fermat-Type Matrix Equation over $\GL_2(\mathbb Z)$}
\author{
Hongjian Li$^{1,}$\footnote{E-mail\,$:$ lhj@gdufs.edu.cn.}\quad
Weilin Zhang$^{2,}$\footnote{Corresponding author. E-mail\,$:$ weilin@m.scnu.edu.cn.}\\
{\small\it $^{1}$School of Mathematics and Statistics, Guangdong University of Foreign Studies,}\\
{\small\it Guangzhou 510006, Guangdong, P. R. China}\\
{\small\it $^{2}$School of Mathematics and Information Science, Guangzhou University,}\\
{\small\it Guangzhou 510006, Guangdong, P. R. China}
}
\date{}
\maketitle
\date{}

\noindent{\bf Abstract}\quad
We determine the complete ordered solution set of the Fermat-type matrix
equation $X^n+Y^n=Z^n$ in $\GL_2(\mathbb Z)$ for every integer
$n\ge3$, expressed in terms of simultaneous integral-conjugacy orbits
while keeping the variable symmetries of the equation separate. The
equation is solvable if and only if $4\nmid n$ and $6\nmid n$. Thus,
in contrast with the determinant-one setting, odd multiples of $3$ are
solvable in the full group $\GL_2(\mathbb Z)$ and necessarily involve
determinant-$-1$ components. For every admissible even exponent, the
full ordered solution set has a canonical disjoint orbit decomposition
indexed by the same transversal that occurs in the matrix Pythagorean
equation. For odd exponents with $3\nmid n$, the commuting solutions
form exactly six simultaneous-conjugacy orbits. The noncommuting odd
solutions form exactly twenty simultaneous-conjugacy orbits when
$3\nmid n$, and exactly fourteen when $3\mid n$. Under the coarser
equivalence generated by signed variable permutations and simultaneous
integral conjugation, these noncommuting orbits collapse to four and
three classes, respectively. The proof combines the canonical orbit
classification of the matrix Pythagorean equation, Cayley--Hamilton
recurrences, commutator divisibility, trace-growth estimates, integral
conjugacy of torsion matrices, and an elementary symmetry-reduced trace
argument for the cubic equation.


\medskip
\noindent{\bf MR(2020) Subject Classification} 11D41; 15A24; 15A36; 20E06

\section{Introduction}\label{sec:introduction}

We study the ordered solutions of
\begin{equation}\label{eq:Fermat}
X^n+Y^n=Z^n,
\qquad X,Y,Z\in\GL_2(\mathbb Z),\quad n\ge3.
\end{equation}
In contrast with the classical scalar equation, the additive relation
among \(X^n\), \(Y^n\), and \(Z^n\) cannot in general be transferred to
a single commutative number field, because the matrices need not commute.
At the same time, the Cayley--Hamilton identity, determinant signs,
centralizer structures, and the restricted torsion of
\(\GL_2(\mathbb Z)\) impose strong arithmetic constraints.

Matrix versions of Fermat's equation have been studied for integral
\(2\times2\) matrices and several related classes; see Chien--Meng,
Grytczuk, Grytczuk--Kurzyd{\l}o, Le--Li, and Qin
\cite{ChienMeng2021,Grytczuk1995,GrytczukKurzydlo2011,LeLi1995,Qin1996}.
Vaserstein considered the problem in \(\GL_2(\mathbb Z)\) from the
viewpoint of noncommutative number theory \cite{Vaserstein1989}, while
Khazanov treated the determinant-one integral \(2\times2\) setting
\cite{Khazanov1995}.  In particular, Qin proved that
\(X^n+Y^n=Z^n\) has a solution in \(\SL_2(\mathbb Z)\) if and only if
\(3\nmid n\) and \(4\nmid n\) \cite{Qin1996}.  More recently,
Chien--Meng studied several structured classes of \(2\times2\) matrices
\cite{ChienMeng2021}, and Sarma gave a general construction of solutions
over \(M_m(\mathbb Z)\) \cite{Sarma2025}.  For the quadratic equation,
Arnold--Eydelzon gave a parametrization of integral matrix Pythagorean
triples \cite{ArnoldEydelzon2019}.

For commuting subsystems, Li--Yuan reduced Fermat- and Catalan-type
matrix equations in a fixed centralizer to Diophantine equations over
quadratic fields \cite{LiYuan2023FermatCatalan}.  The present problem is
more rigid in a different direction: we work in the full group
\(\GL_2(\mathbb Z)\), do not assume commutativity, and classify the
ordered solutions at the level of simultaneous integral-conjugacy
orbits.  The resulting solvability criterion also differs from the
\(\SL_2(\mathbb Z)\) criterion: the determinant-\(-1\) sector supplies
solutions for odd multiples of \(3\).

The main structural point of the present paper is that three different
operations must not be conflated.  First, the variables are ordered.
Second, the equation has sign and permutation symmetries, whose precise
form depends on the parity of \(n\).  Third, \(\GL_2(\mathbb Z)\) acts by
simultaneous conjugation.  We use the variable symmetries only as a
reduction device and then return to the ordered solution set, where the
final classification is expressed as a disjoint union of simultaneous-
conjugacy orbits.  This separation is essential: a union obtained by
applying noncanonical symmetry operations is generally an ordinary
union, whereas a canonical orbit decomposition is a disjoint union.

The commuting odd solutions arise from the unit group of the Eisenstein
integers.  For noncommuting odd solutions, Cayley--Hamilton recurrences,
commutator-divisibility relations, and trace-growth estimates force the
power triple to have finite order.  The remaining torsion problem is
finite, but the passage from the coarse signed-permutation classification
to ordered simultaneous-conjugacy orbits requires an additional orbit-
splitting analysis.  The cubic case lies outside the range of the general
growth estimates and is therefore settled by an elementary
symmetry-reduced small-trace argument.
All arithmetic reductions needed for this exceptional exponent are displayed
in the proof itself; no computer-assisted enumeration or external
computer-algebra verification is used.

For even exponents, the companion matrix Pythagorean classification
\cite{LiZhang2026Pythagorean} provides more than the quartic obstruction.
Its Theorem~1.1, Lemma~2.10, and Corollary~2.1 supply both the residual-orbit
criterion and the canonical transversal needed to classify the relative
position of the Eisenstein and Gaussian structures.  This replaces a
nonunique parametrization involving two
independent conjugating matrices by a genuine disjoint orbit
decomposition.

Throughout the paper, put
\[
G:=\GL_2(\mathbb Z)
\]
and define
\[
\mathcal F_n:=
\{(X,Y,Z)\in G^3:X^n+Y^n=Z^n\}.
\]
Let \(\mathcal F_n^{\mathrm{com}}\) be the subset of pairwise commuting
solutions and let
\(\mathcal F_n^{\mathrm{nc}}:=\mathcal F_n\setminus
\mathcal F_n^{\mathrm{com}}\).

The group \(G\) acts on \(\mathcal F_n\) by simultaneous integral
conjugation:
\[
P\cdot(X,Y,Z)
=(PXP^{-1},PYP^{-1},PZP^{-1}).
\]
For \(\mathbf X=(X,Y,Z)\in\mathcal F_n\), write
\[
\mathcal O_G(\mathbf X)
:=\{P\cdot\mathbf X:P\in G\}.
\]
These orbits are pairwise disjoint and partition \(\mathcal F_n\).
Here and throughout the paper, \(\cup\) denotes an ordinary union,
which may have overlaps, whereas \(\sqcup\) denotes a disjoint union.
We use \(\sqcup\) only after the corresponding orbit representatives
have been proved inequivalent.

When \(n\) is odd, set \(u_1=X\), \(u_2=Y\), and \(u_3=-Z\).  Then
\(u_1^n+u_2^n+u_3^n=0\).  Hence \(S_3\) acts on \(\mathcal F_n\) by
permuting \((u_1,u_2,u_3)\).  We denote the induced transformation by
\[
\rho_\pi(X,Y,Z)
:=\bigl(u_{\pi^{-1}(1)},u_{\pi^{-1}(2)},-u_{\pi^{-1}(3)}\bigr),
\qquad \pi\in S_3.
\]
In particular,
\begin{equation}\label{eq:rho-generators}
\begin{aligned}
\rho_{12}(X,Y,Z)&=(Y,X,Z),\\
\rho_{23}(X,Y,Z)&=(X,-Z,-Y),\\
\rho_{13}(X,Y,Z)&=(-Z,Y,-X).
\end{aligned}
\end{equation}
The overall sign involution
\(\iota(X,Y,Z)=(-X,-Y,-Z)\) also preserves the equation.  Thus the
odd equation carries an auxiliary symmetry group
\(\Gamma:=\langle\rho_\pi,\iota\rangle\cong S_3\times C_2\).
The \(\Gamma\)-action commutes with simultaneous conjugation.  We shall
use it to normalize determinant positions, but the final statements are
made for the ordered set \(\mathcal F_n\), not merely modulo \(\Gamma\).

Fix
\begin{equation}\label{eq:basic-matrices}
Q=
\begin{pmatrix}
1&-1\\
1&0
\end{pmatrix},
\qquad
R=Q^2=
\begin{pmatrix}
0&-1\\
1&-1
\end{pmatrix},
\qquad
J=
\begin{pmatrix}
0&-1\\
1&0
\end{pmatrix}.
\end{equation}
Then
\begin{equation}\label{eq:basic-relations}
Q^6=\Id,
\qquad R^3=\Id,
\qquad R+R^2=-\Id,
\qquad J^2=-\Id,
\qquad Q+Q^{-1}=\Id.
\end{equation}

The canonical transversal imported from the Pythagorean orbit
classification is
\begin{equation}\label{eq:T-transversal-intro}
\mathcal T
:=\left\{
\begin{pmatrix}
a&-a-c\\
a+b&-a
\end{pmatrix}:
(a,b,c)\in\mathbb Z^3,
\ ab+bc+ca=1,
\ a\le b,
\ a<c
\right\}.
\end{equation}
Every \(T\in\mathcal T\) satisfies \(T^2=-\Id\).

For the odd noncommuting part, let
\begin{equation}\label{eq:four-base-triples}
\begin{aligned}
\mathbf A&:=
\left(
\smat{-1&0\\-1&1},
\smat{0&1\\1&0},
\smat{-1&1\\0&1}
\right),\\[1ex]
\mathbf B&:=
\left(
\smat{-1&0\\-1&1},
J,
\smat{-1&-1\\0&1}
\right),\\[1ex]
\mathbf C_3&:=
\left(
\smat{-1&0\\0&1},
R,
\smat{-1&-1\\1&0}
\right),\\[1ex]
\mathbf C_4&:=
\left(
\smat{-1&-1\\0&1},
J,
\smat{-1&-2\\1&1}
\right).
\end{aligned}
\end{equation}
Their ordered component-order patterns are, respectively,
\((2,2,2)\), \((2,4,2)\), \((2,3,3)\), and \((2,4,4)\).

Let \(n\) be odd.  If \(\mathbf T=(T_1,T_2,T_3)\) is one of the
admissible torsion triples occurring below, define
\begin{equation}\label{eq:rt-definition}
\operatorname{rt}_n(\mathbf T)
:=\bigl(T_1^{e_1},T_2^{e_2},T_3^{e_3}\bigr),
\qquad
ne_i\equiv1\pmod{\ord(T_i)}.
\end{equation}
For an involution we take \(e_i=1\).  The map is defined for
\(\mathbf A\), \(\mathbf B\), and \(\mathbf C_4\) for every odd
\(n\), and for \(\mathbf C_3\) exactly when \(3\nmid n\).
Lemma~\ref{lem:finite-root-recovery} will show that these are the unique
integral \(n\)th roots of the corresponding nonscalar torsion
components.

Define the following finite set of ordered root triples:
\begin{equation}\label{eq:canonical-odd-reps}
\begin{aligned}
\mathscr R_n^{\mathrm{nc}}
:={}&
\{\epsilon\,\operatorname{rt}_n(\mathbf A):
\epsilon\in\{\pm1\}\}\\
&\cup
\{\epsilon\,\rho\operatorname{rt}_n(\mathbf B):
\epsilon\in\{\pm1\},\
\rho\in\{1,\rho_{12},\rho_{23}\}\}\\
&\cup
\{\epsilon\,\rho\operatorname{rt}_n(\mathbf C_4):
\epsilon\in\{\pm1\},\
\rho\in\{1,\rho_{12},\rho_{13}\}\}\\
&\cup
\begin{cases}
\{\rho_\pi\operatorname{rt}_n(\mathbf C_3):\pi\in S_3\},
&3\nmid n,\\
\varnothing,&3\mid n.
\end{cases}
\end{aligned}
\end{equation}
Here multiplication by \(\epsilon\) means an overall sign change of all
three components.  The unions in the definition of
\(\mathscr R_n^{\mathrm{nc}}\) are ordinary set unions; the theorem
below asserts, in particular, that the displayed triples are pairwise
inequivalent under simultaneous integral conjugation.

\begin{theorem}
\label{thm:full-classification}
Let \(n\ge3\).
\begin{enumerate}[label=\textup{(\roman*)}]
\item The solution set \(\mathcal F_n\) is nonempty if and only if
\(4\nmid n\) and \(6\nmid n\).

\item If \(n\equiv2,10\pmod{12}\), then
\begin{equation}\label{eq:even-main-decomposition}
\mathcal F_n
=
\mathop{\bigsqcup}\limits_{\epsilon_1,\epsilon_2\in\{\pm1\}}
\mathop{\bigsqcup}\limits_{T\in\mathcal T}
\mathcal O_G(-\epsilon_1R^2,-\epsilon_2R,T).
\end{equation}

\item If \(n\) is odd and \(3\nmid n\), let
\(e\in\mathbb Z/6\mathbb Z\) be determined by
\(ne\equiv1\pmod6\).  Then
\begin{equation}\label{eq:commuting-main-decomposition}
\mathcal F_n^{\mathrm{com}}
=
\mathop{\bigsqcup}\limits_{s\in\mathbb Z/6\mathbb Z}
\mathcal O_G(Q^{s+e},Q^{s-e},Q^s).
\end{equation}
If \(n\) is odd and \(3\mid n\), then
\(\mathcal F_n^{\mathrm{com}}=\varnothing\).

\item If \(n\) is odd, then
\begin{equation}\label{eq:odd-noncommuting-main-decomposition}
\mathcal F_n^{\mathrm{nc}}
=
\mathop{\bigsqcup}\limits_{\mathbf X\in
\mathscr R_n^{\mathrm{nc}}}
\mathcal O_G(\mathbf X).
\end{equation}
Consequently, \(\mathcal F_n^{\mathrm{nc}}\) consists of exactly
\(20\) simultaneous-conjugacy orbits when \(3\nmid n\), and exactly
\(14\) such orbits when \(3\mid n\).  Under the coarser equivalence
generated by \(\Gamma\) and simultaneous conjugation, these orbits
collapse to four classes represented by
\(\operatorname{rt}_n(\mathbf A),\operatorname{rt}_n(\mathbf B),
\operatorname{rt}_n(\mathbf C_3),\operatorname{rt}_n(\mathbf C_4)\)
when \(3\nmid n\), and to the three classes represented by
\(\operatorname{rt}_n(\mathbf A),\operatorname{rt}_n(\mathbf B),
\operatorname{rt}_n(\mathbf C_4)\) when \(3\mid n\).
\end{enumerate}
\end{theorem}

\begin{corollary}\label{thm:main}
Let \(n\ge3\).  Equation~\eqref{eq:Fermat} has a solution in
\(\GL_2(\mathbb Z)\) if and only if \(4\nmid n\) and \(6\nmid n\).
\end{corollary}

The paper is organized into four sections.  Section~\ref{sec:preliminaries}
collects the torsion, quadratic-order, Cayley--Hamilton, commutator,
growth, centralizer, and root-recovery tools.  Section~\ref{sec:main-proof}
proves Theorem~\ref{thm:full-classification}.  More precisely,
Subsection~\ref{sec:solvability} determines the solvable exponents,
Subsection~\ref{sec:finite-orbits} classifies the noncommuting torsion
triples and their ordered simultaneous-conjugacy orbits,
Subsection~\ref{sec:odd-classification} treats odd exponents, and
Subsection~\ref{sec:even-classification} treats the admissible even
exponents; the final subsection assembles these results into the proof
of the theorem.  The exceptional cubic exponent is handled entirely in
the main text by an explicit four-pattern arithmetic reduction.

\section{Preliminaries}\label{sec:preliminaries}

\subsection{Torsion matrices and quadratic orders}

\begin{lemma}\label{lem:finite-orders}
If $T\in\GL_2(\mathbb Z)$ has finite order, then
$\ord(T)\in\{1,2,3,4,6\}$. Moreover, if $\det T=-1$, then
$T^2=\Id$; if $\ord(T)=4$, then $T^2=-\Id$; and if
$\ord(T)=6$, then $T^3=-\Id$.
\end{lemma}

\begin{proof}
A finite-order matrix over a field of characteristic zero is
diagonalizable, and all of its eigenvalues are roots of unity. Since the
characteristic polynomial of $T$ lies in $\mathbb Z[x]$ and has degree
$2$, the total degree of the cyclotomic factors that occur is at most
$2$. The cyclotomic polynomials of degree at most $2$ are
\[
\begin{aligned}
\Phi_1(x)&=x-1,
&\Phi_2(x)&=x+1,
&\Phi_3(x)&=x^2+x+1,\\
\Phi_4(x)&=x^2+1,
&\Phi_6(x)&=x^2-x+1.
\end{aligned}
\]
Hence the order of $T$ is one of $1,2,3,4,6$. If $\det T=-1$,
the preceding cyclotomic classification shows that the characteristic
polynomial cannot be $\Phi_3,\Phi_4,$ or $\Phi_6$, nor can it be
$(x-1)^2$ or $(x+1)^2$. Thus the only possibility is
$(x-1)(x+1)=x^2-1$, and $T^2=\Id$. Finally,
$\Phi_4(T)=0$ gives $T^2=-\Id$, while $\Phi_6(T)=0$ gives
$T^2-T+\Id=0$ and hence $T^3=-\Id$.
\end{proof}

\begin{lemma}\label{lem:quadratic-euclidean}
The rings $\mathbb Z[i]$ and $\mathbb Z[\omega]$, where
$\omega^2+\omega+1=0$, are Euclidean domains with respect to the norms
$N(a+bi)=a^2+b^2$ and $N(a+b\omega)=a^2-ab+b^2$, respectively.
Consequently, both are principal ideal domains.
\end{lemma}

\begin{proof}
First consider $\mathbb Z[i]$. Given
$z=x+yi\in\mathbb Q(i)$, choose $m,n\in\mathbb Z$ such that
$|x-m|\le\frac12$ and $|y-n|\le\frac12$. With $q=m+ni$, we have
$N(z-q)=(x-m)^2+(y-n)^2\le\frac12<1$. Thus, for
$\alpha,\beta\in\mathbb Z[i]$ with $\beta\ne0$, apply this estimate to
$z=\alpha/\beta$ and put $r=\alpha-q\beta$. Then
$\alpha=q\beta+r$ and
$N(r)=N(\beta)N(z-q)<N(\beta)$. Hence $\mathbb Z[i]$ is Euclidean with
respect to $N$.

Now consider $\mathbb Z[\omega]$. Given
$z=x+y\omega\in\mathbb Q(\omega)$, choose $m,n\in\mathbb Z$ such that
$|x-m|\le\frac12$ and $|y-n|\le\frac12$. Put
$a=x-m$, $b=y-n$, and $q=m+n\omega$. Then $|a|,|b|\le1/2$ and
$N(z-q)=a^2-ab+b^2$. On the square $|a|,|b|\le1/2$, this quadratic
form is at most $\frac14+\frac14+\frac14=\frac34<1$. Therefore, for
all $\alpha,\beta\in\mathbb Z[\omega]$ with $\beta\ne0$, there exist
$q,r\in\mathbb Z[\omega]$ such that
$\alpha=q\beta+r$ and $N(r)<N(\beta)$. Thus $\mathbb Z[\omega]$ is
also Euclidean. Every Euclidean domain is a principal ideal domain.
\end{proof}

\subsection{Cayley--Hamilton recurrences and commutator divisibility}

For $\varepsilon\in\{\pm1\}$, define a sequence of polynomials by
\begin{equation}\label{eq:U-recurrence}
U_0(t,\varepsilon)=0,
\qquad U_1(t,\varepsilon)=1,
\qquad
U_{k+1}(t,\varepsilon)=tU_k(t,\varepsilon)-\varepsilon U_{k-1}(t,\varepsilon).
\end{equation}

\begin{lemma}\label{lem:CH-expansion}
Let $M\in\GL_2(\mathbb Z)$, and set
$t=\tr M$ and $\varepsilon=\det M$. For every $n\ge1$,
\begin{equation}\label{eq:CH-power}
M^n=U_n(t,\varepsilon)M-
\varepsilon U_{n-1}(t,\varepsilon)\Id.
\end{equation}
Moreover, if $d_n=\tr(M^n)$, then
\begin{equation}\label{eq:trace-discriminant}
d_n^2-4\varepsilon^n
=(t^2-4\varepsilon)U_n(t,\varepsilon)^2.
\end{equation}
\end{lemma}

\begin{proof}
The Cayley--Hamilton identity $M^2-tM+\varepsilon\Id=0$, together
with the recurrence~\eqref{eq:U-recurrence}, proves
\eqref{eq:CH-power} by induction on $n$.

Let $\lambda,\mu$ be the eigenvalues of $M$ in an algebraic closure.
If $\lambda\ne\mu$, then
$U_n(t,\varepsilon)=(\lambda^n-\mu^n)/(\lambda-\mu)$ and
$d_n=\lambda^n+\mu^n$. Therefore
\[
\begin{aligned}
d_n^2-4\varepsilon^n
&=(\lambda^n+\mu^n)^2-4(\lambda\mu)^n\\
&=(\lambda^n-\mu^n)^2\\
&=(\lambda-\mu)^2U_n(t,\varepsilon)^2\\
&=(t^2-4\varepsilon)U_n(t,\varepsilon)^2.
\end{aligned}
\]
If $\lambda=\mu$, then $t^2=4\varepsilon$,
$U_n(t,\varepsilon)=n\lambda^{n-1}$, and $d_n=2\lambda^n$. Both sides
of~\eqref{eq:trace-discriminant} are then zero. Thus the identity holds
in all cases.
\end{proof}

Put $A=X^n$, $B=Y^n$, and $C=Z^n$, so that $A+B=C$.
Denote by $p,q,r$ the coefficients of $X,Y,Z$, respectively, in the
expansion~\eqref{eq:CH-power}. Then $[A,B]=pq[X,Y]$. Since
$[A,C]=[A,B]$ and $[B,C]=-[A,B]$, we also have
$[A,C]=pr[X,Z]$ and $[B,C]=qr[Y,Z]$.

\begin{lemma}\label{lem:comm-divisibility}
Let $\Delta=\det[A,B]$. Then the following three exact identities hold:
\begin{equation}\label{eq:three-commutator-identities}
\Delta
=p^2q^2\det[X,Y]
=p^2r^2\det[X,Z]
=q^2r^2\det[Y,Z].
\end{equation}
In particular, if $pqr\ne0$, they imply the divisibility relations
\begin{equation}\label{eq:three-divisibilities}
p^2q^2\mid\Delta,
\qquad
p^2r^2\mid\Delta,
\qquad
q^2r^2\mid\Delta.
\end{equation}
\end{lemma}

\begin{proof}
By the Cayley--Hamilton expansion,
$[A,B]=pq[X,Y]$, $[A,C]=pr[X,Z]$, and $[B,C]=qr[Y,Z]$.
On the other hand, $C=A+B$ gives
$[A,C]=[A,A+B]=[A,B]$ and
$[B,C]=[B,A+B]=[B,A]=-[A,B]$. Since the matrices are $2\times2$,
$\det(-[A,B])=\det[A,B]=\Delta$. Taking determinants yields
\eqref{eq:three-commutator-identities}. If $pqr\ne0$, the three
determinants on the right are integers, and hence
\eqref{eq:three-divisibilities}.
\end{proof}

We also need the following determinant identity for commutators of
$2\times2$ matrices.

\begin{lemma}\label{lem:comm-det-formula}
For all $M,N\in M_2(\mathbb Z)$, one has
\begin{align}\label{eq:comm-det-general}
\det(MN-NM)
={}&4\det M\det N
-\det M\,(\tr N)^2
-\det N\,(\tr M)^2\notag\\
&-(\tr(MN))^2
+\tr M\tr N\tr(MN).
\end{align}
\end{lemma}

\begin{proof}
Put $M_0=M-(\tr M)\Id/2$ and $N_0=N-(\tr N)\Id/2$. Then $[M,N]=[M_0,N_0]$, and both $M_0,N_0$ have trace zero.
For any trace-zero $2\times2$ matrices $S,T$, the Cayley--Hamilton
theorem gives
\[
S^2=-\det(S)\Id,
\qquad
T^2=-\det(T)\Id,
\qquad
ST+TS=\tr(ST)\Id.
\]
Therefore
\[
(ST-TS)^2
=\bigl(\tr(ST)^2-4\det(S)\det(T)\bigr)\Id.
\]
Since $ST-TS$ also has trace zero,
\[
\det(ST-TS)=4\det(S)\det(T)-\tr(ST)^2.
\]
Finally, substitute
\[
\begin{aligned}
\det(M_0)&=\det M-\frac{(\tr M)^2}{4},
&\det(N_0)&=\det N-\frac{(\tr N)^2}{4},\\
\tr(M_0N_0)&=\tr(MN)-\frac{\tr M\tr N}{2}.
\end{aligned}
\]
Expanding gives~\eqref{eq:comm-det-general}.
\end{proof}

\subsection{Trace estimates, centralizers, and root constraints}

Unless otherwise stated, throughout this subsection $n\ge5$ is odd. We
continue to write $A=X^n$, $B=Y^n$, $C=Z^n$, and
$a=\tr A$, $b=\tr B$, $c=\tr C=a+b$.

\begin{lemma}\label{lem:three-trace-bound}
Let $a,b,c\in\mathbb R$ with $c=a+b$. Choose from $a,b,c$ the two
numbers of largest absolute value, denote them by $u,v$, and assume
$|u|\ge|v|$. If $(a,b,c)\ne(0,0,0)$, then
$0<|v|\le|u|\le2|v|$. Moreover, each of the following six quadratic
forms, as well as its negative, satisfies
$|F(a,b)|\le3|uv|$:
\[
\begin{gathered}
a^2+ab+b^2,\quad a^2+3ab+b^2,\quad a^2+ab-b^2,\\
a^2-ab-b^2,\quad -a^2+ab+b^2,\quad -a^2-ab+b^2.
\end{gathered}
\]
\end{lemma}

\begin{proof}
Since $c=a+b$, the largest of $|a|,|b|,|c|$ is at most the sum
of the other two, so $|u|\le2|v|$. Any two distinct elements
$s,t\in\{a,b,c\}$ satisfy $|st|\le|uv|$, and every
$s\in\{a,b,c\}$ satisfies $s^2\le u^2\le2|uv|$. Using also
\[
\begin{gathered}
a^2+3ab+b^2=c^2+ab,\qquad a^2+ab-b^2=ac-b^2,\\
a^2-ab-b^2=a^2-bc,\qquad -a^2+ab+b^2=-a^2+bc,\\
-a^2-ab+b^2=-ac+b^2,\qquad a^2+ab+b^2=ac+b^2,
\end{gathered}
\]
the assertion follows by applying the triangle inequality to each form.
\end{proof}

\begin{lemma}\label{lem:growth-minus}
Let $n\ge5$ be odd, and let $M\in\GL_2(\mathbb Z)$ satisfy
$\det M=-1$ and $\tr M\ne0$. Put
$d=\tr(M^n)$ and $u=U_n(\tr M,-1)$. Then
$|d|\ge11, \; u^2\ge2|d|+3$.
\end{lemma}

\begin{proof}
Let $s=|\tr M|\ge1$ and
$\lambda=(s+\sqrt{s^2+4})/2$. Then $\lambda>1$, and the other
eigenvalue is $-\lambda^{-1}$. Since $n$ is odd, $|d|=\lambda^n-\lambda^{-n}$ and
$|u|=(\lambda^n+\lambda^{-n})/(\lambda+\lambda^{-1})$.
For fixed $s$, both quantities increase with odd $n$. Hence
$|d|\ge\lambda^5-\lambda^{-5}=s^5+5s^3+5s\ge11$.

We now prove the second inequality. Define
\[
M_s=
\begin{pmatrix}
s&1\\
1&0
\end{pmatrix}.
\]
Then $\tr M_s=s$ and $\det M_s=-1$, and its eigenvalues are
$\lambda$ and $-\lambda^{-1}$. In particular, for odd $k$, $\tr(M_s^k)=\lambda^k-\lambda^{-k}$.
On the other hand, the recurrence~\eqref{eq:U-recurrence} gives
$U_k(-s,-1)=(-1)^{k-1}U_k(s,-1)$.
Thus, for odd $k$, the trace of the $k$th power of a
negative-determinant matrix with trace $-s$ is the negative of
$\tr(M_s^k)$, while the corresponding coefficient $U_k$ is unchanged.
It therefore suffices, with $s=|\tr M|$, to consider
$H_k(s)=U_k(s,-1)^2-2\tr(M_s^k)-3$. Using
$\lambda-\lambda^{-1}=s$ and
$U_k(s,-1)=(\lambda^k+\lambda^{-k})/(\lambda+\lambda^{-1})$,
and simplifying $H_{k+2}(s)-H_k(s)$ over a common denominator, we obtain
the following identity.  For a direct verification, multiply both sides by
$\lambda^{2k+2}(\lambda^2+1)$ and repeatedly replace
$\lambda^2$ by $s\lambda+1$; after expansion the two sides reduce to the
same polynomial in $s$ and $\lambda$:
\begin{equation}\label{eq:growth-minus-difference}
H_{k+2}(s)-H_k(s)
=
\frac{(\lambda-1)(\lambda+1)(\lambda^{2k+2}+1)
\bigl(\lambda^{2k+2}-2\lambda^{k+2}-2\lambda^k-1\bigr)}
{\lambda^{2k+2}(\lambda^2+1)}.
\end{equation}
For $s\ge1$, we have $\lambda\ge\varphi=(1+\sqrt5)/2$. Since
$k\ge5$,
\[
\begin{aligned}
&\lambda^{2k+2}-2\lambda^{k+2}-2\lambda^k-1\\
&\qquad=\lambda^k\bigl(\lambda^{k+2}-2(\lambda^2+1)\bigr)-1\\
&\qquad\ge \lambda^k\bigl(\lambda^7-2\lambda^2-2\bigr)-1.
\end{aligned}
\]
The function $f(\lambda)=\lambda^7-2\lambda^2-2$ is strictly
increasing for $\lambda\ge\varphi$, and $f(\varphi)>1$; also
$\lambda^k>1$. Thus the displayed expression is strictly positive, so
$H_{k+2}(s)-H_k(s)>0$. Moreover,
$H_5(s)=(s-1)(s^7+s^6+7s^5+5s^4+16s^3+6s^2+12s+2)\ge0$.
Hence $H_n(s)\ge0$, that is, $u^2\ge2|d|+3$.
\end{proof}

\begin{lemma}\label{lem:growth-plus}
Let $n\ge5$ be odd and $M\in\SL_2(\mathbb Z)$. Put
$d=\tr(M^n)$ and $u=U_n(\tr M,1)$. Then
$u=0$ implies $M^n=\pm\Id$. If $u\ne0$, then either
$\tr M\in\{0,\pm1\}$, $|u|=1$, and $|d|\le1$, or
$u^2\ge2|d|+3$.
\end{lemma}

\begin{proof}
Let $t=\tr M$. For $t=0,\pm1$, the recurrence gives the following
periodic patterns. If $t=0$, then
$u=(-1)^{(n-1)/2}$ and $d=0$. If $t=1$, then
\[
(u,d)=
\begin{cases}
(1,1),&n\equiv1\pmod6,\\
(0,-2),&n\equiv3\pmod6,\\
(-1,1),&n\equiv5\pmod6;
\end{cases}
\]
If $t=-1$, then
\[
(u,d)=
\begin{cases}
(1,-1),&n\equiv1\pmod6,\\
(0,2),&n\equiv3\pmod6,\\
(-1,-1),&n\equiv5\pmod6.
\end{cases}
\]
Consequently, in these three cases $t\in\{0,\pm1\}$, if $u=0$,
the Cayley--Hamilton expansion gives
$M^n=-U_{n-1}(t,1)\Id$. Since $M^n\in\SL_2(\mathbb Z)$,
$U_{n-1}(t,1)^2=\det(M^n)=1$,
and hence $M^n=\pm\Id$. If $u\ne0$, then $|u|=1$ and
$|d|\le1$. For $t=\pm2$, we have $|u|=n$ and $|d|=2$, so
$u\ne0$ and $u^2\ge25>7$. If $|t|\ge3$, the eigenvalue formula below
shows that $u\ne0$.

Now suppose $|t|=s\ge3$ and put
$\lambda=(s+\sqrt{s^2-4})/2>1$. Then
$|d|=\lambda^n+\lambda^{-n}$ and
$|u|=(\lambda^n-\lambda^{-n})/(\lambda-\lambda^{-1})$.
Define
\[
N_s=
\begin{pmatrix}
s&-1\\
1&0
\end{pmatrix}.
\]
Then $\tr N_s=s$, $\det N_s=1$, and the eigenvalues of $N_s$ are
$\lambda$ and $\lambda^{-1}$. Therefore $\tr(N_s^k)=\lambda^k+\lambda^{-k}$. At the same time,
the recurrence~\eqref{eq:U-recurrence} gives
$U_k(-s,1)=(-1)^{k-1}U_k(s,1)$.
For odd $k$, the trace of the $k$th power of a
positive-determinant matrix with trace $-s$ is the negative of
$\tr(N_s^k)$, while the corresponding $U_k$ coefficient is unchanged.
It is therefore enough to consider
$K_k(s)=U_k(s,1)^2-2\tr(N_s^k)-3$. Using
$\lambda+\lambda^{-1}=s$ and
$U_k(s,1)=(\lambda^k-\lambda^{-k})/(\lambda-\lambda^{-1})$,
and simplifying $K_{k+2}(s)-K_k(s)$ over a common denominator, we get
the following identity.  For a direct verification, multiply both sides by
$\lambda^{2k+2}(\lambda-1)(\lambda+1)$ and repeatedly replace
$\lambda^2$ by $s\lambda-1$; after expansion the two sides reduce to the
same polynomial in $s$ and $\lambda$:
\begin{equation}\label{eq:growth-plus-difference}
K_{k+2}(s)-K_k(s)
=
\frac{(\lambda^{k+1}-1)(\lambda^{k+1}+1)\,\Xi_k(\lambda)}
{\lambda^{2k+2}(\lambda-1)(\lambda+1)}.
\end{equation}
Here
\[
\Xi_k(\lambda)
=
\lambda^{2k+4}+\lambda^{2k+2}
-2\lambda^{k+4}+4\lambda^{k+2}-2\lambda^k+\lambda^2+1.
\]
Put $T=\lambda^k$. Then
$\Xi_k(\lambda)=T\bigl(\lambda^2(\lambda^2+1)T-2(\lambda^2-1)^2\bigr)
+\lambda^2+1$.
Since $\lambda\ge(3+\sqrt5)/2>2$ and $k\ge5$, we have
$T\ge\lambda^5$, and therefore
$\lambda^2(\lambda^2+1)T\ge\lambda^7(\lambda^2+1)>2\lambda^4
>2(\lambda^2-1)^2$. Thus the expression in parentheses is strictly
positive, and hence $\Xi_k(\lambda)>0$. Furthermore,
$K_5(s)=s^8-6s^6-2s^5+11s^4+10s^3-6s^2-10s-2$. Its derivative is
$2(s^2-s-1)(s^2+s-1)(4s^3-6s-5)>0$ for $s\ge3$, and
$K_5(3)=2776>0$. Hence $K_n(s)>0$.
\end{proof}

\begin{lemma}\label{lem:common-eigenvector}
Let $S,T\in M_2(\overline{\mathbb Q})$. Then $\det[S,T]=0$ if and only
if $S$ and $T$ have a common eigenvector.
\end{lemma}

\begin{proof}
If $S$ and $T$ have a common eigenvector, they can be simultaneously
upper triangularized. Their commutator is then strictly upper triangular,
and its determinant is zero.

Conversely, the assertion is immediate if $S$ is scalar. If $S$ has
two distinct eigenvalues, write, in an eigenbasis for $S$,
\[
S=\begin{pmatrix}\lambda_1&0\\0&\lambda_2\end{pmatrix},
\qquad
T=\begin{pmatrix}p&q\\r&s\end{pmatrix}.
\]
Then $\det[S,T]=(\lambda_1-\lambda_2)^2qr$. Hence $q=0$ or
$r=0$, and the two matrices have a common eigenvector. If $S$ has a
single eigenvalue but is not scalar, write
\[
S=\lambda\Id+\begin{pmatrix}0&1\\0&0\end{pmatrix}.
\]
A direct calculation gives $\det[S,T]=-r^2$. Thus $r=0$, and again
the matrices have a common eigenvector.
\end{proof}

\begin{lemma}\label{lem:nonscalar-centralizer}
If $M\in M_2(\mathbb Q)$ is nonscalar, then
\[
C_{M_2(\mathbb Q)}(M)=\mathbb Q[M]
=\{\alpha\Id+\beta M:\alpha,\beta\in\mathbb Q\}.
\]
In particular, if a nonscalar matrix $N$ commutes with $M$, then
$C_{M_2(\mathbb Q)}(N)=C_{M_2(\mathbb Q)}(M)$.
\end{lemma}

\begin{proof}
Because $M$ is nonscalar, its minimal polynomial has degree $2$. There
is a vector $v\in\mathbb Q^2$ for which $v,Mv$ are linearly independent;
otherwise $Mv\in\mathbb Qv$ for every $v$, forcing $M$ to be scalar.

Take $T\in C_{M_2(\mathbb Q)}(M)$. Relative to the basis $v,Mv$,
write $Tv=\alpha v+\beta Mv$ with $\alpha,\beta\in\mathbb Q$. Since
$TM=MT$, we have
$T(Mv)=M(Tv)=\alpha Mv+\beta M^2v$. Thus $T$ and
$\alpha\Id+\beta M$ agree on the basis $v,Mv$, and hence
$T=\alpha\Id+\beta M$. The reverse inclusion is immediate, proving
$C_{M_2(\mathbb Q)}(M)=\mathbb Q[M]$.

If the nonscalar matrix $N$ commutes with $M$, then
$\mathbb Q[N]\subseteq C_{M_2(\mathbb Q)}(M)=\mathbb Q[M]$. Both sides
are two-dimensional $\mathbb Q$-algebras, so $\mathbb Q[N]=\mathbb Q[M]$.
Applying the first part again yields
$C_{M_2(\mathbb Q)}(N)=\mathbb Q[N]=\mathbb Q[M]
=C_{M_2(\mathbb Q)}(M)$.
\end{proof}

\begin{lemma}\label{lem:SL2-singular-commutator}
Let $A,B\in\SL_2(\mathbb Z)$. Then $\det[A,B]=0$ if and only if
$[A,B]=0$.
\end{lemma}

\begin{proof}
Sufficiency is clear. For necessity, if $A$ or $B$ is scalar, the
conclusion is immediate; assume that both are nonscalar. By
Lemma~\ref{lem:common-eigenvector}, $A$ and $B$ have a common eigenline
$L$ in $\overline{\mathbb Q}^{\,2}$.

If $L$ is defined over $\mathbb Q$, choose a primitive integral vector
$v\in\mathbb Z^2$ spanning $L$ and extend it to a basis of
$\mathbb Z^2$. Relative to this integral basis, both $A$ and $B$ are
upper triangular integral matrices. Write
\[
A=\begin{pmatrix}\alpha&r\\0&\alpha'\end{pmatrix},
\qquad
B=\begin{pmatrix}\beta&s\\0&\beta'\end{pmatrix}.
\]
Since $\alpha,\alpha',\beta,\beta'\in\mathbb Z$ and
$\alpha\alpha'=\beta\beta'=1$, we have
$\alpha=\alpha'=\epsilon$ and $\beta=\beta'=\delta$ for some
$\epsilon,\delta\in\{\pm1\}$. Hence
\[
A=\begin{pmatrix}\epsilon&r\\0&\epsilon\end{pmatrix},
\qquad
B=\begin{pmatrix}\delta&s\\0&\delta\end{pmatrix}.
\]
Direct multiplication now gives $AB=BA$.

Suppose instead that $L$ is not defined over $\mathbb Q$. Since $A$ is
nonscalar, $L$ is an eigenspace of $A$. The characteristic polynomial of
$A$ has degree $2$, so the field of definition of $L$ has degree at most
$2$. Nonrationality implies that it is a quadratic field $K$. Let $L'$ be
the conjugate of $L$ under the nontrivial element of
$\operatorname{Gal}(K/\mathbb Q)$. Because the entries of $A$ and $B$ lie
in $\mathbb Q$, preservation of $L$ implies preservation of $L'$.
Moreover $L\ne L'$, so $L\oplus L'=K^2$, and $A,B$ are both diagonal in
the same basis with coordinate axes $L,L'$. Thus $AB=BA$.

In either case, $\det[A,B]=0$ implies $[A,B]=0$.
\end{proof}

\begin{lemma}\label{lem:golden-root}
Let $\varphi=(1+\sqrt5)/2$ and $K=\mathbb Q(\sqrt5)$. Let $n\ge3$ be
odd. If an algebraic number $\theta$ satisfies
$\theta^n\in\{\pm\varphi,\pm\varphi^{-1}\}$, then
$[\mathbb Q(\theta):\mathbb Q]\ge2n$. In particular, $\theta$ cannot
belong to the compositum of two quadratic fields.
\end{lemma}

\begin{proof}
First we prove that
\[
\mathcal O_K=\mathbb Z[\varphi],
\qquad
\mathcal O_K^\times=\{\pm\varphi^m:m\in\mathbb Z\}.
\]
The first identity is the standard description of the ring of integers
of the quadratic field of discriminant $5$. For the second, let
$\epsilon$ be a unit. After replacing it by its negative or inverse, we
may assume that its chosen real embedding satisfies $\epsilon\ge1$.
Choose $m\in\mathbb Z$ such that
$\varphi^m\le\epsilon<\varphi^{m+1}$, and set
$\delta=\epsilon\varphi^{-m}$. Then $\delta$ is a unit and
$1\le\delta<\varphi$. Write
$\delta=(u+v\sqrt5)/2$, where $u,v\in\mathbb Z$ and $u\equiv v\pmod2$.
If $N_{K/\mathbb Q}(\delta)=1$, the conjugate embedding is
$\delta^{-1}$, so
$2\le u=\delta+\delta^{-1}<\varphi+\varphi^{-1}=\sqrt5<3$.
Thus $u=2$, and the norm equation $u^2-5v^2=4$ gives $v=0$, hence
$\delta=1$. If the norm is $-1$, the conjugate embedding is
$-\delta^{-1}$, whence
$0\le u=\delta-\delta^{-1}<\varphi-\varphi^{-1}=1$. Thus $u=0$.
The norm equation $u^2-5v^2=-4$ would then give $5v^2=4$, which is
impossible for $v\in\mathbb Z$. Hence the norm-$-1$ case cannot occur.
Therefore $\delta=1$, and every unit is of the form $\pm\varphi^m$.

Set $\eta=\theta^n=\pm\varphi^e$, where $e=\pm1$. Since $\eta$ is
not rational, $K=\mathbb Q(\eta)\subseteq\mathbb Q(\theta)$. Let
$d=[K(\theta):K]$. The element $\theta$ is a root of the monic polynomial
$T^n-\eta\in\mathcal O_K[T]$, so it is integral over $\mathcal O_K$ and
hence an algebraic integer. Moreover,
$\theta^{-1}=\eta^{-1}\theta^{n-1}$ is also an algebraic integer, so
$\theta$ is a unit. Since $\theta$ satisfies a polynomial of degree $n$
over $K$, we have $d\le n$. Hence
$N_{K(\theta)/K}(\theta)=\pm\varphi^k$ for some $k\in\mathbb Z$.
Taking norms in $\theta^n=\eta$ gives
$(\pm\varphi^k)^n=(\pm\varphi^e)^d$. Comparing exponents of $\varphi$
yields $nk=ed$. Since $e=\pm1$, we have $n\mid d$. Together with
$d\le n$, this gives $d=n$, and therefore
$[\mathbb Q(\theta):\mathbb Q]=[K(\theta):K][K:\mathbb Q]=2n$.
\end{proof}

\begin{lemma}\label{lem:mixed-delta-nonzero}
Let $n\ge3$ be odd and suppose that $X^n+Y^n=Z^n$. Assume that
$A=X^n$, $B=Y^n$, and $C=Z^n$ are all nonscalar and that
$[A,B]\ne0$.

\begin{enumerate}[label=\textup{(\roman*)}]
\item if $\det A=\det C=-1$ and $\det B=1$, then
$\det[A,B]\ne0$;
\item if $\det A=-1$ and $\det B=\det C=1$, then
$\det[A,B]\ne0$.
\end{enumerate}
\end{lemma}

\begin{proof}
Suppose, to the contrary, that $\det[A,B]=0$. By
Lemma~\ref{lem:common-eigenvector}, $A$ and $B$ have a common eigenline
$\ell$. Since $A,B,C=A+B$ are nonscalar and
$XA=AX$, $YB=BY$, $ZC=CZ$, Lemma~\ref{lem:nonscalar-centralizer} gives
$X\in\mathbb Q[A]$, $Y\in\mathbb Q[B]$, and $Z\in\mathbb Q[C]$.
Thus $X,Y,Z$ all preserve $\ell$. Choosing a basis whose first vector is
a nonzero vector of $\ell$, we may write $X,Y,Z$ simultaneously in upper
triangular form.

Let $\lambda,\mu,\nu$ be their eigenvalues on the common eigenline,
and put $\Lambda=\lambda^n$, $M=\mu^n$, and $N=\nu^n$.
The first diagonal entries give $\Lambda+M=N$.

In case \textup{(i)}, the second diagonal entries give
$-\Lambda^{-1}+M^{-1}=-N^{-1}$. Eliminating $N=\Lambda+M$ yields
$\Lambda^2+\Lambda M-M^2=0$. Consequently,
$\left(\mu/\lambda\right)^n$ is a root of $T^2-T-1$, and hence belongs
to $\{\varphi,-\varphi^{-1}\}$.

In case \textup{(ii)}, the second diagonal entries give
$-\Lambda^{-1}+M^{-1}=N^{-1}$, whence
$\Lambda^2-\Lambda M-M^2=0$. Thus
$\left(\mu/\lambda\right)^n\in\{\varphi^{-1},-\varphi\}$.

On the other hand, $\lambda$ and $\mu$ are eigenvalues of integral
$2\times2$ matrices, so each has degree at most $2$. Hence
$\mu/\lambda$ belongs to a number field of degree at most $4$, contrary
to Lemma~\ref{lem:golden-root}.
\end{proof}

\begin{lemma}\label{lem:odd-scalar-power}
Let $T\in\GL_2(\mathbb Z)$, let $n$ be odd, and suppose that
$T^n=\sigma\Id$ with $\sigma\in\{\pm1\}$. Then $T$ has finite order.
If $3\nmid n$, then $T=\sigma\Id$. If $3\mid n$, the only additional
possibilities are: when $\sigma=1$, $T$ is a nonscalar matrix of order
$3$; when $\sigma=-1$, $T$ is a nonscalar matrix of order $6$.
\end{lemma}

\begin{proof}
From $T^n=\sigma\Id$ we get $T^{2n}=\Id$, so $T$ has finite order.
By the classification of integral torsion matrices in dimension two, its
order belongs to $\{1,2,3,4,6\}$. If $\sigma=1$, then $T^n=\Id$, so
$\ord(T)\mid n$. Since $n$ is odd, either $\ord(T)=1$, or, when
$3\mid n$, $\ord(T)=3$. The first case gives
$T=\Id$.

If $\sigma=-1$, then $T^n=-\Id$. Orders $1$ and $3$ are impossible.
For order $4$, an odd power is still $\pm T$ and cannot be scalar. If
$\ord(T)=2$, then $T^n=T$, so necessarily $T=-\Id$. If
$\ord(T)=6$, then $T^3=-\Id$, and $T^n=-\Id$ precisely when
$n\equiv3\pmod6$, that is, when $3\mid n$.
\end{proof}

\begin{lemma}\label{lem:positive-torsion-normal}
Let $T\in\SL_2(\mathbb Z)$ be a nonscalar torsion matrix. According as
$\tr T=-1,0,1$, the matrix $T$ is integrally conjugate to $R,J,Q$,
respectively. Moreover,
\[
C_{\GL_2(\mathbb Z)}(R)=C_{\GL_2(\mathbb Z)}(Q)=\langle Q\rangle,
\qquad
C_{\GL_2(\mathbb Z)}(J)=\langle J\rangle.
\]
\end{lemma}

\begin{proof}
First suppose that $\tr T=-1$. Then $T^2+T+\Id=0$. Let
$\mathcal O=\mathbb Z[\omega]$, where $\omega^2+\omega+1=0$. Defining
$\omega\cdot v=Tv$ makes $\mathbb Z^2$ an $\mathcal O$-module. It is
clearly finitely generated. If $0\ne\alpha\in\mathcal O$ and
$\alpha v=0$, then $\alpha\ne0$ in $K=\mathbb Q(\omega)$, and the
operator $\alpha(T)$ is invertible on $\mathbb Q^2$, so $v=0$; hence the
module is torsion-free. After tensoring with $\mathbb Q$, the space
$\mathbb Q^2$ is one-dimensional over $K$, so the $\mathcal O$-module
has rank one. By Lemma~\ref{lem:quadratic-euclidean}, $\mathcal O$ is a
principal ideal domain; every finitely generated torsion-free rank-one
module over $\mathcal O$ is therefore free of rank one. Choose an
$\mathcal O$-basis vector $e$. Then $e,\omega e$ is an integral basis of
$\mathbb Z^2$, and multiplication by $\omega$ has matrix
\[
\begin{pmatrix}0&-1\\1&-1\end{pmatrix}=R.
\]
Thus $T$ is integrally conjugate to $R$.

If $\tr T=0$, the same argument with $\mathbb Z[i]$, where $i^2=-1$,
shows that $T$ is integrally conjugate to
\[
J=\begin{pmatrix}0&-1\\1&0\end{pmatrix}.
\]
If $\tr T=1$, put $\zeta=1+\omega$, where
$\omega^2+\omega+1=0$. Then $\zeta^2-\zeta+1=0$,
$\mathbb Z[\zeta]=\mathbb Z[\omega]$, and $\zeta-1=\omega$.
Again $\mathbb Z^2$ is a free rank-one module over
$\mathbb Z[\zeta]$. Choose a $\mathbb Z[\zeta]$-basis vector $e$. Since $\{1,\zeta-1\}=\{1,\omega\}$ is a $\mathbb Z$-basis of
$\mathbb Z[\zeta]=\mathbb Z[\omega]$, the vectors $e$ and
$(\zeta-1)e$ form an integral basis of $\mathbb Z^2$. Moreover,
$\zeta e=e+(\zeta-1)e$ and $\zeta(\zeta-1)e=-e$,
so multiplication by $\zeta$ has matrix
\[
Q=\begin{pmatrix}1&-1\\1&0\end{pmatrix}.
\]
Hence $T$ is integrally conjugate to $Q$. We now compute the
centralizers. Let $S\in\GL_2(\mathbb Z)$ commute with $R$. Relative to
the free rank-one $\mathcal O$-module structure above, $S$ is an
$\mathcal O$-module automorphism, and therefore is multiplication by a
unit $\alpha\in\mathcal O^\times$. Conversely, every unit gives an
integral module automorphism. Hence
$C_{\GL_2(\mathbb Z)}(R)\cong\mathcal O^\times$. The Eisenstein norm is
$N(a+b\omega)=a^2-ab+b^2$, and the units are exactly the elements of
norm $1$. Since
$4N(a+b\omega)=(2a-b)^2+3b^2=4$, we have $|b|\le1$; checking
$b=-1,0,1$ gives exactly six solutions. Thus
$\mathcal O^\times=\{\pm1,\pm\omega,\pm\omega^2\}$, corresponding to
$\{\Id,Q,Q^2,Q^3,Q^4,Q^5\}=\langle Q\rangle$. Similarly,
$\mathbb Z[i]^\times=\{\pm1,\pm i\}$, so
$C_{\GL_2(\mathbb Z)}(J)=\langle J\rangle$. Finally, $Q$ and
$R=Q^2$ define the same Eisenstein-order action, and hence
$C_{\GL_2(\mathbb Z)}(Q)=\langle Q\rangle$.
\end{proof}

\begin{lemma}\label{lem:torsion-root-centralizer}
Let $A,T\in\GL_2(\mathbb Z)$, where $A$ is a nonscalar torsion matrix,
and suppose that $T^m=A$ for some $m\ge1$. Then $T$ has finite order,
$TA=AT$, and $T\in C_{M_2(\mathbb Q)}(A)=\mathbb Q[A]$. Moreover, if
$m$ is odd and $3\nmid m$, and if $d=\ord(A)$, then there is a unique
residue class $e\in\mathbb Z/d\mathbb Z$ satisfying $me\equiv1\pmod d$,
and $T=A^e$.
\end{lemma}

\begin{proof}
From $T^m=A$ and $A^d=\Id$, where $d=\ord(A)$, we get
$T^{md}=A^d=\Id$, so $T$ has finite order. Moreover,
$TA=T^{m+1}=AT$. Since $A$ is nonscalar,
Lemma~\ref{lem:nonscalar-centralizer} gives
$T\in C_{M_2(\mathbb Q)}(A)=\mathbb Q[A]$.

Now assume that $m$ is odd and $3\nmid m$, and put $D=\ord(T)$.
By Lemma~\ref{lem:finite-orders}, $D\in\{1,2,3,4,6\}$. Because $m$ is
odd and $3\nmid m$, $\gcd(D,m)=1$, and consequently
$d=\ord(T^m)=D/\gcd(D,m)=D$. Thus $T$ and $A=T^m$ generate the same
finite cyclic group, $\langle T\rangle=\langle A\rangle$. Choose the
unique residue class $e\in\mathbb Z/d\mathbb Z$ such that
$me\equiv1\pmod d$. Then $T=A^e$.
\end{proof}

\begin{lemma}\label{lem:power-commuting-lift}
Let $n\ge5$ be odd and suppose that $X^n+Y^n=Z^n$. If
$[X^n,Y^n]=0$, then $X,Y,Z$ commute pairwise.
\end{lemma}

\begin{proof}
Put $A=X^n$, $B=Y^n$, and $C=Z^n=A+B$. First assume that
$A,B,C$ are all nonscalar. From $[A,B]=0$ and
Lemma~\ref{lem:nonscalar-centralizer},
$B\in C_{M_2(\mathbb Q)}(A)=\mathbb Q[A]$, and hence
$C=A+B\in\mathbb Q[A]$. Since $B$ and $C$ are nonscalar, each of them,
together with $\Id$, spans the two-dimensional algebra $\mathbb Q[A]$.
Therefore $\mathbb Q[B]=\mathbb Q[A]=\mathbb Q[C]$. Since $X$ commutes
with $A=X^n$, $Y$ with $B=Y^n$, and $Z$ with $C=Z^n$, we have
$X\in C(A)=\mathbb Q[A]$,
$Y\in C(B)=\mathbb Q[B]=\mathbb Q[A]$, and
$Z\in C(C)=\mathbb Q[C]=\mathbb Q[A]$.
Thus $X,Y,Z$ all lie in the commutative algebra $\mathbb Q[A]$, and
therefore commute pairwise.

It remains to consider the case in which at least one of $A,B,C$ is
scalar. By applying an element of $\Gamma$, we may
assume that $A$ is scalar. Write $A=k\Id$ with $k\in\mathbb Z$. Since
$A=X^n\in\GL_2(\mathbb Z)$, $1=|\det A|=k^2$, so $k=\pm1$. We may
therefore write $A=\sigma\Id$ with $\sigma\in\{\pm1\}$.
Then $C=B+\sigma\Id$. Put $\beta=\det B$ and $\gamma=\det C$.
Since
$\gamma=\det(B+\sigma\Id)=\beta+\sigma\tr B+1$, we obtain
$\tr B=\sigma(\gamma-\beta-1)\in\{\pm1,\pm3\}$. If
$\det Y=-1$, then $\det B=-1$. We first show that $\tr Y\ne0$.
Indeed, if $\tr Y=0$, the Cayley--Hamilton theorem gives $Y^2=\Id$.
Since $n$ is odd, $B=Y^n=Y$, so $\tr B=0$, contradicting
$\tr B\in\{\pm1,\pm3\}$. Hence $\tr Y\ne0$. Lemma~\ref{lem:growth-minus}
then gives $|\tr B|=|\tr(Y^n)|\ge11$, another contradiction. Thus
$\det Y=1$, and hence $\beta=\det B=1$.

Let $q=U_n(\tr Y,1)$. If $q=0$, the Cayley--Hamilton expansion gives
$B=Y^n=-U_{n-1}(\tr Y,1)\Id$,
so $B$ is an integral scalar matrix.
Since $\det B=1$, we have $B=\pm\Id$, and therefore $\tr B=\pm2$,
contrary to $\tr B\in\{\pm1,\pm3\}$.

Thus $q\ne0$. By Lemma~\ref{lem:growth-plus}, there are two
possibilities. In the growth branch,
$q^2\ge2|\tr B|+3$. On the other hand, the trace-discriminant identity
gives
$(\tr B)^2-4=((\tr Y)^2-4)q^2$, so $q^2$ divides $(\tr B)^2-4$.
If $|\tr B|=3$, then $q^2\ge9$ but $q^2\mid5$, impossible. If
$|\tr B|=1$, then $q^2\ge5$ but $q^2\mid3$, also impossible. Thus we
must be in the torsion branch of Lemma~\ref{lem:growth-plus}:
$|q|=1$ and $|\tr B|\le1$. Since
$\tr B\in\{\pm1,\pm3\}$, we get $|\tr B|=1$. Finally,
$\tr B=\sigma(\gamma-2)$ gives $\gamma=1$ and
$\sigma\tr B=-1$.

If $\sigma=1$, then $\tr B=-1$, so
$B^2+B+\Id=0$ and $C=B+\Id=-B^2$. Thus $B$ has order $3$, while
$-B^2$ has order $6$, and both lie in the cyclic group $\langle-B\rangle$
of order $6$. If $\sigma=-1$, then $\tr B=1$, so
$B^2-B+\Id=0$ and $C=B-\Id=B^2$. Here $B$ has order $6$ and
$C=B^2$ has order $3$; again both lie in the same cyclic group of order
$6$.

If $3\mid n$, this scalar branch cannot occur. Indeed, if
$Y^n=B$, then $B$ is torsion and
$Y^{n\,\ord(B)}=\Id$, so $Y$ itself is torsion, with order in
$\{1,2,3,4,6\}$. When $n$ is odd and $3\mid n$, the $n$th powers of
matrices of these five possible orders can have orders only
$1,2,1,4,2$, respectively; in particular they cannot be nonscalar of
order $3$. Likewise, no $n$th power of an integral torsion matrix can
have order $6$, so it cannot equal the nonscalar matrix of order $6$
above. If $3\nmid n$, then $\gcd(n,6)=1$. Let $G$ be the cyclic group
of order $6$ containing $B$ and $C$. Since $B$ and $C$ are nonscalar
torsion matrices and $Y^n=B$, $Z^n=C$,
Lemma~\ref{lem:torsion-root-centralizer} gives, more precisely,
\[
Y=B^{e_B},\qquad Z=C^{e_C},
\]
where $ne_B\equiv1\pmod{\ord(B)}$ and
$ne_C\equiv1\pmod{\ord(C)}$. In particular, $Y,Z\in G$ and hence
$YZ=ZY$. By Lemma~\ref{lem:odd-scalar-power},
$X^n=\sigma\Id$ gives $X=\sigma\Id$. Thus $X,Y,Z$ commute pairwise.
The other possible positions of the scalar power reduce to this one by
applying elements of $\Gamma$.
\end{proof}

\section{Proof of Theorem~\ref{thm:full-classification}}\label{sec:main-proof}

\subsection{Solvable exponents}\label{sec:solvability}

We first determine exactly which exponents can occur.  The two
obstructions, divisibility by $4$ and by $6$, are proved separately;
afterward an explicit construction gives solutions for every remaining
exponent.  The case $4\mid n$ is reduced to the companion
Pythagorean classification.

We use the following consequence of the ordered orbit classification of
the matrix Pythagorean equation.  It is precisely here that the
canonical transversal \(\mathcal T\) from
\eqref{eq:T-transversal-intro} enters the Fermat problem.

\begin{proposition}
\label{prop:pythagorean-orbit-input}
Let
\[
\mathcal P_+
:=\{(A,B,C)\in\SL_2(\mathbb Z)^3:A^2+B^2=C^2\}.
\]
Then
\begin{equation}\label{eq:Pplus-decomposition}
\mathcal P_+
=
\mathop{\bigsqcup}\limits_{\epsilon_1,\epsilon_2\in\{\pm1\}}
\mathop{\bigsqcup}\limits_{T\in\mathcal T}
\mathcal O_G(-\epsilon_1R^2,-\epsilon_2R,T).
\end{equation}
In particular, every \((A,B,C)\in\mathcal P_+\) satisfies
\(C^2=-\Id\), and hence \(C\) has order \(4\).
\end{proposition}

\begin{proof}
This is the determinant-one specialization of Theorem~1.1, together with
the square-root conjugacy statement of Lemma~2.3, the residual-equivalence
criterion of Lemma~2.10, and the transversal statement of Corollary~2.1, in
\cite{LiZhang2026Pythagorean}.  We record the short deduction needed here.
Choose signs so that the three traces are nonnegative and apply that orbit
theorem.  In that theorem, the
three determinant patterns of the canonical pieces are
\[
(1,1,1),\qquad(-1,1,1),\qquad(1,-1,1).
\]
Since \(A,B,C\in\SL_2(\mathbb Z)\), only the first piece can occur.
Thus the trace-normalized triple belongs to an orbit of
\((-R^2,-R,T_0)\) with \(T_0^2=-\Id\).  Restoring the first two signs
produces the two sign parameters in \eqref{eq:Pplus-decomposition}.
Restoring the sign of the third component does not create an additional
sign layer.  Indeed, if \(T_0\in\mathcal T\), then \(-T_0\) is again an
integral square root of \(-\Id\).  Lemma~2.3 of
\cite{LiZhang2026Pythagorean} therefore places \(-T_0\) in the same
integral conjugacy class \(\mathcal T_J\) as \(J\).  Corollary~2.1 of the
same paper then gives a unique \(T_1\in\mathcal T\) and some
\(k\in\{0,1,2\}\) such that
\[
T_1=R^k(-T_0)R^{-k}.
\]
Since \(R^k\) fixes both \(-\epsilon_1R^2\) and
\(-\epsilon_2R\) under conjugation, it follows that
\[
\mathcal O_G(-\epsilon_1R^2,-\epsilon_2R,-T_0)
=
\mathcal O_G(-\epsilon_1R^2,-\epsilon_2R,T_1).
\]
Thus the third sign merely permutes the parameter set \(\mathcal T\).

The union is disjoint.  Different pairs \((\epsilon_1,\epsilon_2)\)
have different ordered trace pairs, because
\(\tr(-\epsilon_1R^2)=\epsilon_1\) and
\(\tr(-\epsilon_2R)=\epsilon_2\).  Fix a sign pair and suppose that two
parameters \(T_1,T_2\in\mathcal T\) determine the same simultaneous-conjugacy
orbit.  The central signs on the first two components do not affect the
condition that a conjugating matrix centralize \(R\).  Hence Lemma~2.10 of
\cite{LiZhang2026Pythagorean} implies that \(T_1\) and \(T_2\) lie in the
same \(\langle R\rangle\)-orbit inside \(\mathcal T_J\).  Corollary~2.1
of the same paper says that each such residual orbit meets \(\mathcal T\)
in exactly one point, so \(T_1=T_2\).  Finally, every representative in
\eqref{eq:Pplus-decomposition} has third component squaring to
\(-\Id\), which proves the last assertion.
\end{proof}

\begin{remark}
\label{rem:companion-input}
The numbering above refers to the version of \cite{LiZhang2026Pythagorean}
used throughout this paper.  Theorem~1.1 is the trace-normalized ordered
orbit decomposition, Lemma~2.3 is the conjugacy classification of integral
square roots of \(-\Id\), Lemma~2.10 identifies equality of the first-family
\(G\)-orbits with the residual \(\langle R\rangle\)-action on
\(\mathcal T_J\), and Corollary~2.1 states that \(\mathcal T\) is a
complete transversal for those residual orbits.  These are the only parts
of the companion classification used in Proposition~\ref{prop:pythagorean-orbit-input}.
\end{remark}

\begin{proposition}
\label{prop:no-4-divides}
If \(4\mid n\), then Equation~\eqref{eq:Fermat} has no solution in
\(\GL_2(\mathbb Z)\).
\end{proposition}

\begin{proof}
Write \(n=4m\) and suppose that
\(X^n+Y^n=Z^n\).  Set
\[
A=X^{2m},\qquad B=Y^{2m},\qquad C=Z^{2m}.
\]
Then \(A,B,C\in\SL_2(\mathbb Z)\) and
\(A^2+B^2=C^2\).  Proposition~\ref{prop:pythagorean-orbit-input}
gives \(C^2=-\Id\).  Since \(C=Z^{2m}\), we obtain
\[
(Z^m)^4=-\Id.
\]
Hence \(Z^m\) has order \(8\), contradicting
Lemma~\ref{lem:finite-orders}.  Therefore no solution exists when
\(4\mid n\).
\end{proof}

\begin{remark}
The displayed decomposition \eqref{eq:Pplus-decomposition}, rather than
only the consequence \(C^2=-\Id\), will be used again in
Subsection~\ref{sec:even-classification}.  At this stage, the distinction between an ordinary parameter union and
a canonical disjoint orbit union becomes essential.
\end{remark}

\begin{lemma}\label{lem:modular-group-presentation}
Let $s=\overline{J}$ and $r=\overline{R}$ be the images of $J,R$ in
$\PSL_2(\mathbb Z)$. Then
$\PSL_2(\mathbb Z)\cong\langle s,r\mid s^2=r^3=1\rangle\cong C_2*C_3$.
\end{lemma}

\begin{proof}
This is the classical structure theorem for the modular group; see, for
example, Serre's standard treatment of group actions on trees
\cite[Chapter~I.4]{SerreTrees}. For the precise form used here, consider
the action of $\PSL_2(\mathbb Z)$ on the barycentric subdivision of the
dual tree of the Farey tessellation. A fundamental edge joins a vertex
with stabilizer generated by $s=\overline J$ to a vertex with stabilizer
generated by $r=\overline R$; these stabilizers are isomorphic to $C_2$
and $C_3$, respectively. The edge stabilizer is trivial and the quotient
graph is a single edge. Bass--Serre theory therefore gives
$\PSL_2(\mathbb Z)\cong C_2*C_3$. We shall use only this presentation and
the natural quotient homomorphism
$C_2*C_3\twoheadrightarrow C_3$.
\end{proof}

We next exclude multiples of $6$.  Here the obstruction comes from the
modular group rather than from the Pythagorean classification.

\begin{proposition}\label{prop:no-6-divides}
If $6\mid n$, then Equation~\eqref{eq:Fermat} has no solution.
\end{proposition}

\begin{proof}
Write $n=2m$ with $3\mid m$, and set
$A=X^n$, $B=Y^n$, and $C=Z^n$. Since $n$ is even,
$A,B,C\in\SL_2(\mathbb Z)$. Put $U=A^{-1}B$. From $A+B=C$ and
$\det A=\det B=\det C=1$, we obtain
$\det U=1$ and $\det(\Id+U)=1$. Hence
$\tr U=-1$ and $U^2+U+\Id=0$. Thus
$\Id+U=-U^2$, $B=AU$, and $C=A(\Id+U)=-AU^2$.

Let $\omega^2+\omega+1=0$. The relation
$U^2+U+\Id=0$ makes $\mathbb Z^2$ a finitely generated
$\mathbb Z[\omega]$-module. This module is torsion-free: if
$0\ne\alpha\in\mathbb Z[\omega]$ and $\alpha v=0$, then the determinant
of $\alpha(U)$ on $\mathbb Q^2$ equals
$N_{\mathbb Q(\omega)/\mathbb Q}(\alpha)\ne0$, so $v=0$. Its
$\mathbb Z$-rank is $2=[\mathbb Q(\omega):\mathbb Q]$, hence its rank
over $\mathbb Z[\omega]$ is $1$. By
Lemma~\ref{lem:quadratic-euclidean}, $\mathbb Z[\omega]$ is a principal
ideal domain, so the module is free of rank one. Therefore there is
$P\in\GL_2(\mathbb Z)$ such that $P^{-1}UP=R$. After simultaneously
conjugating $A,B,C,U$, we may assume that $U=R$.

By Lemma~\ref{lem:modular-group-presentation}, write
$\PSL_2(\mathbb Z)\cong C_2*C_3=\langle a,r\mid a^2=r^3=1\rangle$.
Bars will denote images in $\PSL_2(\mathbb Z)$, and we put
$g=\overline A$ and $r=\overline R$. Since $n=2m$, we have
$A=(X^2)^m$, $B=(Y^2)^m$, and $C=(Z^2)^m$. As
$X^2,Y^2,Z^2\in\SL_2(\mathbb Z)$, the elements
$\overline A,\overline B,\overline C$ are all $m$th powers in
$\PSL_2(\mathbb Z)$. On the other hand, from $B=AR$ and $C=-AR^2$,
and because $-\Id$ is trivial in $\PSL_2(\mathbb Z)$, we get
$\overline A=g$, $\overline B=gr$, and $\overline C=gr^2$.
Thus $g,gr,gr^2$ are all $m$th powers. Consider the homomorphism
$\pi:\PSL_2(\mathbb Z)\to C_3$ defined by $\pi(a)=0$ and
$\pi(r)=1$, with $C_3$ written additively. If $h=k^m$, then
$\pi(h)=m\pi(k)=0$ because $3\mid m$. Hence
$\pi(g)=\pi(gr)=\pi(gr^2)=0$. But homomorphy also gives
$\pi(gr)-\pi(g)=\pi(r)=1$, a contradiction. Notice that no matrix
addition has been projected to the modular group; only the multiplicative
relations $B=AR$ and $C=-AR^2$, which were deduced from $A+B=C$, are
used.
\end{proof}

It remains to prove that the two divisibility obstructions above are the
only ones.  The required solutions can be written down explicitly.

\begin{proposition}\label{prop:existence-construction}
If $4\nmid n$ and $6\nmid n$, then Equation~\eqref{eq:Fermat} has a
solution in $\GL_2(\mathbb Z)$.
\end{proposition}

\begin{proof}

If $n$ is odd, take
\[
X=
\begin{pmatrix}
-1&0\\
-1&1
\end{pmatrix},
\quad
Y=
\begin{pmatrix}
0&1\\
1&0
\end{pmatrix},
\quad
Z=
\begin{pmatrix}
-1&1\\
0&1
\end{pmatrix}.
\]
A direct calculation gives $X^2=Y^2=Z^2=\Id$ and $X+Y=Z$.
Since $n$ is odd, $X^n=X$, $Y^n=Y$, and $Z^n=Z$, yielding a
solution.

If $n$ is even and $4\nmid n$, $6\nmid n$, then
$n\equiv2$ or $10\pmod{12}$. Hence $\gcd(n,6)=2$, so there exist
integers $u,v$ such that $nu\equiv2\pmod6$ and
$nv\equiv4\pmod6$. Put $X=Q^u$, $Y=Q^v$, and $Z=J$. Since
$Q^6=\Id$, we have $X^n=Q^2=R$ and $Y^n=Q^4=R^2$. Moreover,
$n/2$ is odd, so $Z^n=(J^2)^{n/2}=-\Id$. The relation
$R+R^2=-\Id$ now gives $X^n+Y^n=Z^n$.
\end{proof}

Propositions~\ref{prop:no-4-divides}, \ref{prop:no-6-divides}, and
\ref{prop:existence-construction} therefore establish the exponent
criterion in Corollary~\ref{thm:main}.

\subsection{Noncommuting torsion triples and integral conjugacy}
\label{sec:finite-orbits}

We now work with the ordered torsion solution set
\[
\mathcal N_{\mathrm{tor}}
:=\{(A,B,C)\in G^3:A+B=C,\ A,B,C\text{ torsion},\ [A,B]\ne0\}.
\]
The auxiliary group $\Gamma$ from the Introduction acts on this set,
and its action commutes with simultaneous conjugation.  We first use
$\Gamma$ to normalize determinant positions and obtain eight convenient
triples.  This first reduction gives an ordinary union of $G$-orbits,
because different normalized triples and different elements of
$\Gamma$ may still determine the same simultaneous-conjugacy orbit.
Afterward we split the coarse classes and choose a genuine transversal
for the ordered $G$-orbits.

\subsubsection{Eight normalized triples and four coarse symmetry classes}

\paragraph{Type $\mathrm A_1$.}
\[
A=
\begin{pmatrix}
-1&0\\
-1&1
\end{pmatrix},
\quad
B=
\begin{pmatrix}
0&1\\
1&0
\end{pmatrix},
\quad
C=
\begin{pmatrix}
-1&1\\
0&1
\end{pmatrix}.
\]
\paragraph{Type $\mathrm A_2$.}
\[
A=
\begin{pmatrix}
-1&1\\
0&1
\end{pmatrix},
\quad
B=
\begin{pmatrix}
1&0\\
1&-1
\end{pmatrix},
\quad
C=
\begin{pmatrix}
0&1\\
1&0
\end{pmatrix}.
\]
Both types satisfy $A^2=B^2=C^2=\Id$.

\paragraph{Type $\mathrm B_1$.}
\[
A=
\begin{pmatrix}
-1&0\\
-1&1
\end{pmatrix},
\quad
B=J,
\quad
C=
\begin{pmatrix}
-1&-1\\
0&1
\end{pmatrix}.
\]
\paragraph{Type $\mathrm B_2$.}
\[
A=
\begin{pmatrix}
-1&1\\
0&1
\end{pmatrix},
\quad
B=J,
\quad
C=
\begin{pmatrix}
-1&0\\
1&1
\end{pmatrix}.
\]
Both types satisfy $A^2=C^2=\Id$ and $B^2=-\Id$.

Put
\[
D=
\begin{pmatrix}
-1&0\\
0&1
\end{pmatrix}.
\]
\paragraph{Type $\mathrm C_3$.}
\[
A=D,
\quad
B=R,
\quad
C=
\begin{pmatrix}
-1&-1\\
1&0
\end{pmatrix}.
\]
This type satisfies $A^2=\Id$ and $B^3=C^3=\Id$.

\paragraph{Type $\mathrm C_{4,1}$.}
\[
A=
\begin{pmatrix}
-1&-1\\
0&1
\end{pmatrix},
\quad
B=J,
\quad
C=
\begin{pmatrix}
-1&-2\\
1&1
\end{pmatrix}.
\]
\paragraph{Type $\mathrm C_{4,2}$.}
\[
A=
\begin{pmatrix}
-1&0\\
1&1
\end{pmatrix},
\quad
B=J,
\quad
C=
\begin{pmatrix}
-1&-1\\
2&1
\end{pmatrix}.
\]
Both types satisfy $A^2=\Id$ and $B^2=C^2=-\Id$.

\paragraph{Type $\mathrm C_6$.}
\[
A=D,
\quad
B=Q,
\quad
C=
\begin{pmatrix}
0&-1\\
1&1
\end{pmatrix}.
\]
This type satisfies $A^2=\Id$ and $B^6=C^6=\Id$.

\begin{proposition}\label{prop:standard-triples-check}
Each of the eight triples above satisfies $A+B=C$. Their order patterns,
in the order displayed, are $(2,2,2)$, $(2,2,2)$, $(2,4,2)$,
$(2,4,2)$, $(2,3,3)$, $(2,4,4)$, $(2,4,4)$, and $(2,6,6)$.
Moreover, every one of the eight triples is noncommuting.
\end{proposition}

\begin{proof}
The required assertions follow by direct matrix addition,
multiplication, and commutator calculation. The results are recorded in
the following table; the last column gives $[A,B]=AB-BA$. Every entry in
that column is nonzero, so all eight triples are noncommuting.
\[
\begin{array}{c|c|c|c}
\text{Type}&A+B=C&\text{Order pattern}&[A,B]\\
\hline
\mathrm A_1&\checkmark&(2,2,2)&
\begin{pmatrix}1&-2\\2&-1\end{pmatrix}\\[6pt]
\mathrm A_2&\checkmark&(2,2,2)&
\begin{pmatrix}1&-2\\2&-1\end{pmatrix}\\[6pt]
\mathrm B_1&\checkmark&(2,4,2)&
\begin{pmatrix}-1&2\\2&1\end{pmatrix}\\[6pt]
\mathrm B_2&\checkmark&(2,4,2)&
\begin{pmatrix}1&2\\2&-1\end{pmatrix}\\[6pt]
\mathrm C_3&\checkmark&(2,3,3)&
\begin{pmatrix}0&2\\2&0\end{pmatrix}\\[6pt]
\mathrm C_{4,1}&\checkmark&(2,4,4)&
\begin{pmatrix}-1&2\\2&1\end{pmatrix}\\[6pt]
\mathrm C_{4,2}&\checkmark&(2,4,4)&
\begin{pmatrix}1&2\\2&-1\end{pmatrix}\\[6pt]
\mathrm C_6&\checkmark&(2,6,6)&
\begin{pmatrix}0&2\\2&0\end{pmatrix}
\end{array}
\]
For example, in type $\mathrm C_3$,
\[
D+R=
\begin{pmatrix}
-1&-1\\
1&0
\end{pmatrix}=C,
\qquad
D^2=\Id,
\qquad
R^3=C^3=\Id.
\]
The other rows are obtained by the same direct calculation.
\end{proof}

\begin{theorem}\label{thm:finite-orbits}
Let $A,B,C\in G$ be torsion matrices satisfying $A+B=C$ and
$[A,B]\ne0$.  After applying an element of $\Gamma$ and a simultaneous
integral conjugation, the ordered triple $(A,B,C)$ is represented by one
of
\[
\mathrm A_1,\ \mathrm A_2,\ \mathrm B_1,\ \mathrm B_2,
\ \mathrm C_3,\ \mathrm C_{4,1},\ \mathrm C_{4,2},\ \mathrm C_6.
\]
Equivalently, if $\mathscr S_8$ denotes this set of eight displayed
triples, then
\begin{equation}\label{eq:ordinary-eight-union}
\mathcal N_{\mathrm{tor}}
=
\bigcup_{\gamma\in\Gamma}
\bigcup_{\mathbf T\in\mathscr S_8}
\mathcal O_G(\gamma\mathbf T).
\end{equation}
The two unions in \eqref{eq:ordinary-eight-union} are ordinary unions.
Under the coarser equivalence generated jointly by $\Gamma$ and
simultaneous conjugation, the eight normalized triples form exactly four
classes:
\[
[\mathrm A_1]=[\mathrm A_2],\qquad
[\mathrm B_1]=[\mathrm B_2],\qquad
[\mathrm C_3]=[\mathrm C_6],\qquad
[\mathrm C_{4,1}]=[\mathrm C_{4,2}].
\]
These four coarse classes may be represented by
$\mathbf A$, $\mathbf B$, $\mathbf C_3$, and $\mathbf C_4$ from
\eqref{eq:four-base-triples}.
\end{theorem}

\begin{proof}
If one component were scalar, then the difference of the other two
would be scalar, and the three matrices would commute pairwise, contrary
to the hypothesis. Thus all three components are nonscalar. By
Lemma~\ref{lem:finite-orders}, every component of determinant $-1$ is a
trace-zero involution, while a component of determinant $1$ has trace
$-1,0,$ or $1$, corresponding to order $3,4,$ or $6$, respectively.

We first exclude the case in which all three determinants are positive.
Fix the positive-determinant component $B$ and conjugate it to $R,J,$ or
$Q$ according as its trace is $-1,0,$ or $1$. Write
\[
A=\begin{pmatrix}x&y\\z&a-x\end{pmatrix},
\qquad
a=\tr A\in\{-1,0,1\},
\]
and put $b=\tr B$. Since $c=a+b$ must also lie in
$\{-1,0,1\}$, only the following seven trace combinations need be
considered. The conditions $\det A=1$ and $\det(A+B)=1$ become the
quadratic equations in the table; their left-hand sides are strictly
positive for all integers $x,y$.
\[
\begin{array}{c|c}
(a,b,c)&\text{Resulting equation}\\
\hline
(0,-1,-1)&x^2+xy+y^2-y+1=0\\
(1,-1,0)&x^2+xy-x+y^2-y+1=0\\
(-1,0,-1)&x^2+x+y^2-y+1=0\\
(0,0,0)&x^2+y^2-y+1=0\\
(1,0,1)&x^2-x+y^2-y+1=0\\
(-1,1,0)&x^2+xy+x+y^2+1=0\\
(0,1,1)&x^2+xy+y^2-y+1=0
\end{array}
\]
None of these equations has an integral solution. Indeed, viewed
successively as quadratic equations in $x$, the discriminants of all but
the fourth are
$-(3y^2-4y+4)$, $-(3y^2-2y+3)$, $-(4y^2-4y+3)$,
\[
-(4y^2-4y+3),\quad
-(3y^2-2y+3),\quad
-(3y^2-4y+4).
\]
The quadratic polynomials in parentheses all have negative
discriminant and are therefore strictly positive for every real $y$.
Thus these six quadratic equations in $x$ have no real root. The
left-hand side of the fourth equation is
$x^2+y^2-y+1=x^2+y(y-1)+1\ge1$ for integral $y$, since
$y(y-1)\ge0$. Hence the fourth equation also has no integral solution.

We now classify the triples according to the number of components with
negative determinant.

\medskip
\noindent\textbf{Three negative determinants.}
Here $A,B,C$ are involutions. From $C=A+B$ and $C^2=\Id$ we obtain
$AB+BA=-\Id$. Put $S=AB$. Then $\det S=1$ and $\tr S=-1$, so $S$ has
order $3$. After conjugation we may assume $S=R$. Since $B=AR$, the
condition $B^2=\Id$ is equivalent to $AR=R^2A$. Write
\[
A=\begin{pmatrix}x&y\\z&-x\end{pmatrix}.
\]
Direct multiplication gives
\[
AR-R^2A=
\begin{pmatrix}
x+y-z&0\\
0&x+y-z
\end{pmatrix},
\]
Thus $AR=R^2A$ is equivalent to $z=x+y$. The condition
$\det A=-x^2-yz=-1$ becomes $x^2+xy+y^2=1$. Multiplying by $4$ gives
$(2x+y)^2+3y^2=4$, so $|y|\le1$. Substitution of $y=-1,0,1$ yields
exactly the six pairs
$(x,y)=(-1,0),(1,0),(-1,1),(0,1),(0,-1),(1,-1)$, and hence exactly the
following six matrices:
\[
\begin{pmatrix}-1&0\\-1&1\end{pmatrix},
\quad
\begin{pmatrix}-1&1\\0&1\end{pmatrix},
\quad
\begin{pmatrix}0&-1\\-1&0\end{pmatrix},
\]
\[
\begin{pmatrix}0&1\\1&0\end{pmatrix},
\quad
\begin{pmatrix}1&-1\\0&-1\end{pmatrix},
\quad
\begin{pmatrix}1&0\\1&-1\end{pmatrix}.
\]
Under conjugation by
$C_{\GL_2(\mathbb Z)}(R)=\langle Q\rangle$, the six solutions split into
the following two orbits:
\[
\left\{
\begin{pmatrix}-1&0\\-1&1\end{pmatrix},
\begin{pmatrix}0&1\\1&0\end{pmatrix},
\begin{pmatrix}1&-1\\0&-1\end{pmatrix}
\right\},
\]
\[
\left\{
\begin{pmatrix}-1&1\\0&1\end{pmatrix},
\begin{pmatrix}0&-1\\-1&0\end{pmatrix},
\begin{pmatrix}1&0\\1&-1\end{pmatrix}
\right\}.
\]
This follows directly by computing $Q^jAQ^{-j}$ for
$j=0,\ldots,5$. The corresponding triples are precisely
$\mathrm A_1,\mathrm A_2$.

\medskip
\noindent\textbf{Two negative determinants.}
After applying an element of $\Gamma$, assume that
$\det A=\det C=-1$ and $\det B=1$. The trace relation forces
$\tr B=0$, so after conjugation we may take $B=J$. Write
\[
A=\begin{pmatrix}x&y\\z&-x\end{pmatrix}.
\]
The equations $\det A=-1$ and $\det(A+J)=-1$ give
$y-z=1$ and $x^2+yz=1$. Thus $z=y-1$ and
$x^2+y(y-1)=1$. Since $y(y-1)\ge0$ for every integer $y$, we have
$x^2\le1$. If $x=0$, then $y(y-1)=1$, whose discriminant is $5$, so
there is no integral solution. If $x=\pm1$, then $y(y-1)=0$, and hence
$y=0$ or $1$. Therefore
$(x,y,z)\in\{(-1,0,-1),(-1,1,0),(1,0,-1),(1,1,0)\}$. Under conjugation
by $C_{\GL_2(\mathbb Z)}(J)=\langle J\rangle$, these four solutions split
as follows:
\[
\left\{
\begin{pmatrix}-1&0\\-1&1\end{pmatrix},
\begin{pmatrix}1&1\\0&-1\end{pmatrix}
\right\},
\]
\[
\left\{
\begin{pmatrix}-1&1\\0&1\end{pmatrix},
\begin{pmatrix}1&0\\-1&-1\end{pmatrix}
\right\}.
\]
Thus representatives may be chosen as
\[
\begin{pmatrix}-1&0\\-1&1\end{pmatrix},
\qquad
\begin{pmatrix}-1&1\\0&1\end{pmatrix}.
\]
The resulting triples are exactly $\mathrm B_1$ and $\mathrm B_2$.

\medskip
\noindent\textbf{One negative determinant.}
We may assume that $\det A=-1$ and $\det B=\det C=1$. Then
$\tr A=0$ and $\tr B=\tr C\in\{-1,0,1\}$. Fix successively
$B=R,J,Q$, and again write
\[
A=\begin{pmatrix}x&y\\z&-x\end{pmatrix}.
\]
The conditions $\det A=-1$ and $\det(A+B)=1$ are equivalent to
$x^2+yz=1$ and $\tr(AB)=-1$. For $B=R$ or $B=Q$, a direct calculation
gives $\tr(AB)=x+y-z$, so $z=x+y+1$. Substitution into $x^2+yz=1$
yields $x^2+xy+y^2+y-1=0$. Viewed as a quadratic equation in $x$, its
discriminant is $D_y=4-4y-3y^2$. The condition $D_y\ge0$ gives
$y\in\{-2,-1,0\}$. For $y=-1$, $D_y=5$ is not a square; for $y=-2$,
$D_y=0$ and $x=1$; for $y=0$, $D_y=4$ and $x=\pm1$. Thus
$(x,y,z)=(-1,0,0),(1,-2,0),(1,0,2)$.

For $B=J$, we have $\tr(AJ)=y-z$, so $z=y+1$ and
$x^2+y(y+1)=1$. Since $y(y+1)\ge0$ for integral $y$, we have
$x^2\le1$. If $x=0$, then $y(y+1)=1$ has no integral solution. If
$x=\pm1$, then $y(y+1)=0$, so $y=-1$ or $0$. Hence
$(x,y,z)=(-1,-1,0),(-1,0,1),(1,-1,0),(1,0,1)$. In summary, the complete
sets of integral solutions in the three cases are
\[
\begin{array}{c|c}
B& (x,y,z)\\
\hline
R&(-1,0,0),\ (1,-2,0),\ (1,0,2)\\
J&(-1,-1,0),\ (-1,0,1),\ (1,-1,0),\ (1,0,1)\\
Q&(-1,0,0),\ (1,-2,0),\ (1,0,2)
\end{array}.
\]
We now express the centralizer actions as explicit coordinate
transformations. For
\[
A(x,y,z)=
\begin{pmatrix}x&y\\z&-x\end{pmatrix},
\]
a direct computation gives
\[
Q^{-1}A(x,y,z)Q=A(-x+z,-z,-2x-y+z),\qquad
J^{-1}A(x,y,z)J=A(-x,-z,-y).
\]
Since
\[
C_{\GL_2(\mathbb Z)}(R)
=C_{\GL_2(\mathbb Z)}(Q)=\langle Q\rangle,
\qquad
C_{\GL_2(\mathbb Z)}(J)=\langle J\rangle,
\]
the first transformation cycles
$(-1,0,0)\mapsto(1,0,2)\mapsto(1,-2,0)\mapsto(-1,0,0)$. Thus the three
solutions form a single orbit for each of $B=R$ and $B=Q$. The second
transformation gives
\[
(-1,-1,0)\longleftrightarrow(1,0,1),\qquad
(-1,0,1)\longleftrightarrow(1,-1,0).
\]
Therefore, for $B=J$ there are exactly two orbits:
\[
\{(-1,-1,0),(1,0,1)\},\qquad
\{(-1,0,1),(1,-1,0)\}.
\]
The numbers of orbits in the three cases are $1,2,1$. Their
representatives are
\[
\mathrm C_3,\qquad \mathrm C_{4,1},\qquad
\mathrm C_{4,2},\qquad \mathrm C_6.
\]
This produces eight representatives that keep the negative-determinant
position and the normalized positive-determinant component fixed. Under
the coarser equivalence generated by $\Gamma$ and simultaneous
conjugation, however, these representatives are not all distinct.
Indeed, let
\[
P_A=\begin{pmatrix}-1&1\\-1&0\end{pmatrix},\qquad
P_B=\begin{pmatrix}0&-1\\-1&0\end{pmatrix},\qquad
P_3=\begin{pmatrix}1&0\\0&-1\end{pmatrix}.
\]
Let $\mathcal T_{\mathrm A_1},\ldots,\mathcal T_{\mathrm C_6}$
denote the eight explicit matrix triples above, and let
$P\mathcal T P^{-1}$ mean simultaneous $P$-conjugation of all three
components. Direct calculation gives
\[
\begin{aligned}
\mathcal T_{\mathrm A_2}
&=P_A(-\mathcal T_{\mathrm A_1})P_A^{-1},&
\mathcal T_{\mathrm B_2}
&=P_B(-\mathcal T_{\mathrm B_1})P_B^{-1},\\
\mathcal T_{\mathrm C_{4,2}}
&=P_B(-\mathcal T_{\mathrm C_{4,1}})P_B^{-1}.
\end{aligned}
\]
Starting from $\mathcal T_{\mathrm C_3}=(A,B,C)$, first apply
$\rho_{23}$, which replaces $(A,B,C)$ by $(A,-C,-B)$, and then
conjugate all three components by $P_3$. The result is exactly
$\mathcal T_{\mathrm C_6}$. Hence the eight standard representatives
merge into the four coarse classes stated in the theorem.

Finally, these four coarse classes are pairwise distinct. Type $\mathrm A_1$ has
three negative-determinant components and type $\mathrm B_1$ has two,
whereas the remaining two types have one each. Among the latter,
$\mathrm C_3$ has two positive-determinant components of order $3$ or
$6$, while the two positive-determinant components of
$\mathrm C_{4,1}$ both have order $4$. These properties are invariant
under the allowed operations.
\end{proof}

\subsubsection{Canonical ordered simultaneous-conjugacy orbits}

The preceding theorem deliberately uses the coarse \(\Gamma\)-symmetry
to normalize positions.  We now undo that quotient and select one
representative from every simultaneous-conjugacy orbit of the ordered
torsion equation.  Define
\begin{equation}\label{eq:ordered-torsion-transversal}
\begin{aligned}
\mathscr S_{\mathrm{ord}}
:={}&
\{\epsilon\mathbf A:\epsilon\in\{\pm1\}\}\\
&\cup
\{\epsilon\rho\mathbf B:
\epsilon\in\{\pm1\},\
\rho\in\{1,\rho_{12},\rho_{23}\}\}\\
&\cup
\{\epsilon\rho\mathbf C_4:
\epsilon\in\{\pm1\},\
\rho\in\{1,\rho_{12},\rho_{13}\}\}\\
&\cup
\{\rho_\pi\mathbf C_3:\pi\in S_3\}.
\end{aligned}
\end{equation}

\begin{proposition}
\label{prop:ordered-torsion-orbits}
The set \(\mathscr S_{\mathrm{ord}}\) contains exactly \(20\) triples,
no two of which are simultaneously conjugate in \(G\), and
\begin{equation}\label{eq:ordered-torsion-decomposition}
\mathcal N_{\mathrm{tor}}
=
\mathop{\bigsqcup}\limits_{\mathbf T\in
\mathscr S_{\mathrm{ord}}}
\mathcal O_G(\mathbf T).
\end{equation}
More precisely, the coarse classes represented by
\(\mathbf A,\mathbf B,\mathbf C_3,\mathbf C_4\) split into
\(2,6,6,6\) ordered \(G\)-orbits, respectively.
\end{proposition}

\begin{proof}
Coverage follows from Theorem~\ref{thm:finite-orbits} and the explicit
coarse-class identifications at the end of its proof.  It remains to
identify the exact splitting inside each coarse class.

For the \(\mathbf A\)-class, direct calculation gives
\[
H_{12}\mathbf A H_{12}^{-1}=\rho_{12}\mathbf A,
\qquad
H_{23}\mathbf A H_{23}^{-1}=\rho_{23}\mathbf A,
\]
where
\[
H_{12}=\begin{pmatrix}-1&1\\0&1\end{pmatrix},
\qquad
H_{23}=\begin{pmatrix}-1&0\\-1&1\end{pmatrix}.
\]
Since the transpositions \((12)\) and \((23)\) generate \(S_3\), all
six triples \(\rho_\pi\mathbf A\) lie in
\(\mathcal O_G(\mathbf A)\), and all six
\(-\rho_\pi\mathbf A\) lie in \(\mathcal O_G(-\mathbf A)\).
These two orbits are distinct.  Indeed, solving the simultaneous
intertwining equations
\(P A_i=(-A_i)P\) for the three components of \(\mathbf A\) gives
\[
P=\lambda
\begin{pmatrix}-1&2\\-2&1\end{pmatrix},
\qquad \lambda\in\mathbb Q.
\]
Thus \(\det P=3\lambda^2\), which cannot equal \(\pm1\) for rational
\(\lambda\).  Hence the \(\mathbf A\)-class splits into exactly two
ordered orbits.

In the \(\mathbf B\)-class, the unique order-four component may occupy
any one of the three ordered positions.  These possibilities have
different ordered determinant patterns and hence cannot be conjugate.
For a fixed position, the remaining transposition merely exchanges the
two involutions.  At the base position one has
\[
J\mathbf B J^{-1}=\rho_{13}\mathbf B,
\]
and conjugating this identity by the position-changing operations gives
the corresponding statement in the other two positions.  Thus there is
one orbit for each position and each overall sign.  The two signs do not
merge: the simultaneous intertwining equations between \(\mathbf B\)
and \(-\mathbf B\) give
\[
P=\lambda
\begin{pmatrix}-1&2\\2&1\end{pmatrix},
\qquad \det P=-5\lambda^2,
\]
which again cannot be \(\pm1\) for \(\lambda\in\mathbb Q\).  Therefore
the \(\mathbf B\)-class splits into \(3\cdot2=6\) ordered orbits,
represented by the second line of
\eqref{eq:ordered-torsion-transversal}.

The same argument applies to the \(\mathbf C_4\)-class.  Here the
unique involution may occupy any of the three positions, and the two
order-four components may be exchanged without changing the orbit,
because
\[
H_4\mathbf C_4H_4^{-1}=\rho_{23}\mathbf C_4,
\qquad
H_4=\begin{pmatrix}-1&-1\\0&1\end{pmatrix}.
\]
The intertwiner between \(\mathbf C_4\) and \(-\mathbf C_4\) is again
of the form
\[
\lambda
\begin{pmatrix}-1&2\\2&1\end{pmatrix},
\]
with determinant \(-5\lambda^2\).  Hence the two signs are not
integrally conjugate, and this coarse class also splits into six ordered
orbits.

Finally, the \(\mathbf C_3\)-class contributes exactly the six triples
\(\rho_\pi\mathbf C_3\), \(\pi\in S_3\).  Overall sign introduces no
new orbit, since
\[
J(-\mathbf C_3)J^{-1}=\rho_{23}\mathbf C_3.
\]
The six displayed triples are pairwise nonconjugate because their
ordered determinant and component-order patterns are
\[
\begin{array}{c|c}
\text{determinant pattern}&\text{order pattern}\\
\hline
(-1,1,1)&(2,3,3)\\
(-1,1,1)&(2,6,6)\\
(1,-1,1)&(3,2,3)\\
(1,-1,1)&(6,2,6)\\
(1,1,-1)&(3,6,2)\\
(1,1,-1)&(6,3,2).
\end{array}
\]
Both determinant and order are invariant under simultaneous conjugation.
Thus the \(\mathbf C_3\)-class splits into six ordered orbits.

The four coarse classes cannot meet one another: they have, respectively,
three, two, one, and one negative-determinant components, and in the last
two cases the positive-determinant components have orders in
\(\{3,6\}\) and \(\{4\}\), respectively.  The total number of ordered
orbits is therefore \(2+6+6+6=20\), and
\eqref{eq:ordered-torsion-decomposition} is a disjoint union.
\end{proof}

\subsection{Classification for odd exponents}
\label{sec:odd-classification}

\subsubsection{Commuting solutions}

\begin{theorem}\label{thm:commuting-solutions}
Let $n\ge3$ be odd.  If $3\mid n$, then
$\mathcal F_n^{\mathrm{com}}=\varnothing$.  If $3\nmid n$, let
$e\in\mathbb Z/6\mathbb Z$ satisfy $ne\equiv1\pmod6$.  Then
\begin{equation}\label{eq:commuting-orbit-decomposition}
\mathcal F_n^{\mathrm{com}}
=
\mathop{\bigsqcup}\limits_{s\in\mathbb Z/6\mathbb Z}
\mathcal O_G(Q^{s+e},Q^{s-e},Q^s).
\end{equation}
In particular, the ordered commuting solution set consists of exactly
six simultaneous-conjugacy orbits.
\end{theorem}

\begin{proof}
Assume that $X,Y,Z$ commute pairwise, and put
$S=XZ^{-1}$ and $T=YZ^{-1}$. Then $S^n+T^n=\Id$. Set
$U=S^n$ and $V=T^n=\Id-U$. Let
$\varepsilon=\det U$ and $\delta=\det V$. Since
$\det(\Id-U)=1-\tr U+\det U$, we have
$\tr U=1+\varepsilon-\delta$. The four determinant-sign combinations
give
\[
\begin{array}{c|c|c}
(\varepsilon,\delta)&\tr U&\chi_U(x)\\
\hline
(1,1)&1&x^2-x+1\\
(1,-1)&3&x^2-3x+1\\
(-1,1)&-1&x^2+x-1\\
(-1,-1)&1&x^2-x-1
\end{array}.
\]
The last three characteristic polynomials split over the real quadratic
field $K=\mathbb Q(\sqrt5)$, and in those cases $U$ is nonscalar. Since
$U=S^n$, the matrix $S$ automatically commutes with $U$. By
Lemma~\ref{lem:nonscalar-centralizer},
$S\in C_{M_2(\mathbb Q)}(U)=\mathbb Q[U]\cong K$.
Under this identification, $S$ corresponds to an element
$\alpha\in K$, and $S$ is precisely the matrix of the multiplication
operator $m_\alpha$ on the two-dimensional $\mathbb Q$-vector space $K$.
Thus $N_{K/\mathbb Q}(\alpha)=\det S=\pm1$.
The eigenvalues of $m_\alpha$ are $\alpha$ and its conjugate
$\alpha'$. On the other hand, because $S\in M_2(\mathbb Z)$, its
characteristic polynomial
\[
\chi_S(T)=T^2-\tr(S)T+\det S
\]
is a monic polynomial in $\mathbb Z[T]$, and $\alpha$ is one of its
roots. Hence $\alpha$ is an algebraic integer, so
$\alpha\in\mathcal O_K$. Since
$N_{K/\mathbb Q}(\alpha)=\pm1$, we have
$\alpha\in\mathcal O_K^\times$. Equivalently, the two eigenvalues of
$S$ are the conjugates of a unit of $K$. Put
$\varphi=(1+\sqrt5)/2$. The unit-group calculation in the proof of
Lemma~\ref{lem:golden-root} gives
$\mathcal O_K^\times=\{\pm\varphi^k:k\in\mathbb Z\}$.
Thus either eigenvalue of $S$ is of the form $\pm\varphi^k$. In the
last three cases, the eigenvalue sets of $U$ are respectively
\[
\{\varphi^2,\varphi^{-2}\},\qquad
\{\varphi^{-1},-\varphi\},\qquad
\{\varphi,-\varphi^{-1}\}.
\]
The second and third sets correspond to the characteristic polynomials
$x^2+x-1$ and $x^2-x-1$, respectively. If $U=S^n$, comparison of the
absolute values under either real embedding gives
$nk\in\{\pm2,\pm1\}$. This is impossible because $n\ge3$ and
$k\in\mathbb Z$. Hence none of the last three cases can occur.

We must therefore have $\det U=\det V=1$ and
$U^2-U+\Id=0$. The rank-one freeness of the Eisenstein module implies
that $U$ is integrally conjugate to $Q$ or $Q^{-1}$. After interchanging
$S,T$ if necessary and conjugating simultaneously, assume that $U=Q$ and
$V=Q^{-1}$. Since $S^n=Q$, the matrix $S$ commutes with $Q$; because
$C_{\GL_2(\mathbb Z)}(Q)=\langle Q\rangle$, write $S=Q^e$, with
$ne\equiv1\pmod6$. This congruence is solvable exactly when
$\gcd(n,6)=1$, which for odd $n$ is equivalent to $3\nmid n$.
Similarly, $T=Q^{-e}$.

Since $Z$ commutes with $S$, it also commutes with the nonscalar matrix
$Q=S^n$, and hence $Z=Q^s$. Therefore
$X=SZ=Q^{s+e}$ and $Y=TZ=Q^{s-e}$.

It remains to restore the ordered orientation and prove disjointness.
Let
\[
H=\begin{pmatrix}0&1\\1&0\end{pmatrix}.
\]
Then $HQH^{-1}=Q^{-1}$.  Consequently, the solution obtained by
interchanging the first two components of
$(Q^{s+e},Q^{s-e},Q^s)$ is simultaneously conjugate to the displayed
representative with parameter $-s$.  Thus
\eqref{eq:commuting-orbit-decomposition} contains all ordered commuting
solutions.

Finally, suppose that the representatives with parameters $s$ and $t$
are simultaneously conjugate by $P\in G$.  Taking the quotient of the
first and third components gives
\[
P Q^e P^{-1}=Q^e.
\]
Since $e\equiv\pm1\pmod6$, the matrix $Q^e$ generates
$\langle Q\rangle$, so
$P\in C_G(Q)=\langle Q\rangle$.  Such a matrix fixes every power of
$Q$ under conjugation, and therefore $Q^s=Q^t$, that is,
$s=t$ in $\mathbb Z/6\mathbb Z$.  Hence the union is disjoint.
\end{proof}

\subsubsection{Reduction of noncommuting solutions to torsion}

Throughout this subsection, $n\ge5$ is odd. In the noncommuting case, let
$p,q,r$ again denote the Cayley--Hamilton coefficients of
$X^n,Y^n,Z^n$. By the contrapositive of
Lemma~\ref{lem:power-commuting-lift}, $[A,B]\ne0$. Moreover, $[A,C]=[A,B]$ and $[B,C]=-[A,B]$, so $[A,C]$ and
$[B,C]$ are also nonzero. Together with $[A,B]=pq[X,Y]$,
$[A,C]=pr[X,Z]$, and $[B,C]=qr[Y,Z]$, this implies
$pq\ne0$, $pr\ne0$, and $qr\ne0$,
and hence $pqr\ne0$. Therefore all three divisibility relations in
Lemma~\ref{lem:comm-divisibility} may be used below.

To avoid ambiguity, whenever two particular power traces $u,v$ are
chosen from $a,b,c$, let $M_u,M_v$ denote the corresponding root
matrices, and define
\[
\sigma_u=U_n\bigl(\tr M_u,\det M_u\bigr),
\qquad
\sigma_v=U_n\bigl(\tr M_v,\det M_v\bigr).
\]
Thus $\sigma_u,\sigma_v$ always denote the Cayley--Hamilton
coefficients of the two selected components; they are not functions of
the integers $u,v$ alone.

\begin{proposition}\label{prop:three-negative-complete}
Let $n\ge5$ be odd and suppose that
$\det X=\det Y=\det Z=-1$. If the solution is not pairwise commuting,
then $X^2=Y^2=Z^2=\Id$ and $X+Y=Z$.
\end{proposition}

\begin{proof}
By Lemma~\ref{lem:power-commuting-lift}, $[A,B]\ne0$. Put
$\Delta=\det[A,B]$. The commutator determinant formula gives
$\Delta=a^2+ab+b^2+3>0$. Suppose that $(a,b,c)\ne(0,0,0)$, and choose
the two traces $u,v$ of largest absolute value. The negative-determinant
growth estimate gives $|v|\ge11$, and the two corresponding
Cayley--Hamilton coefficients satisfy
$\sigma_u^2\sigma_v^2\ge(2|u|+3)(2|v|+3)>4|uv|$. On the other hand,
Lemma~\ref{lem:three-trace-bound} gives
$0<\Delta\le3|uv|+3<4|uv|$. This contradicts
$\sigma_u^2\sigma_v^2\mid\Delta$. Hence $a=b=c=0$, so
$A^2=B^2=C^2=\Id$. For example, $X^{2n}=\Id$, so $X$ has finite order;
since $\det X=-1$, Lemma~\ref{lem:finite-orders} gives $X^2=\Id$.
As $n$ is odd, $A=X$. The same argument applies to $Y,Z$.
\end{proof}

\begin{proposition}\label{prop:two-negative-complete}
Let $n\ge5$ be odd and, after applying an element of $\Gamma$, assume that
$\det X=\det Z=-1$ and $\det Y=1$. If the solution is not pairwise
commuting, then $X^2=Z^2=\Id$ and $Y^2=-\Id$.
\end{proposition}

\begin{proof}
By Lemma~\ref{lem:power-commuting-lift}, $[A,B]\ne0$. The commutator
determinant formula gives $\Delta=-a^2-ab+b^2-5$. By
Lemma~\ref{lem:mixed-delta-nonzero}, $\Delta\ne0$.

If $a\ne0$ or $c\ne0$, the negative-determinant growth estimate shows
that one of these traces has absolute value at least $11$. Choose the two
traces $u,v$ of largest absolute value. Then
$|v|\ge|u|/2\ge11/2$, so the two corresponding coefficients are both in
the growth branch and satisfy
$\sigma_u^2\sigma_v^2>4|uv|$. But
Lemma~\ref{lem:three-trace-bound} gives
$0<|\Delta|\le3|uv|+5<4|uv|$, contradicting
$\sigma_u^2\sigma_v^2\mid\Delta$. Thus $a=c=0$ and $b=0$.

Thus $A^2=C^2=\Id$. Since $A=X^n$ and $C=Z^n$, we have
$X^{2n}=Z^{2n}=\Id$. Hence $X$ and $Z$ are torsion matrices. Since
$\det X=\det Z=-1$, Lemma~\ref{lem:finite-orders} gives
$X^2=Z^2=\Id$.
The matrix $B$ satisfies $\tr B=0$ and $\det B=1$, so
$B^2=-\Id$. Since $Y^n=B$, the matrix $Y$ has finite order. Because
$n$ is odd, the only possibility is $\ord(Y)=4$, that is,
$Y^2=-\Id$.
\end{proof}

\begin{proposition}\label{prop:one-negative-complete}
Let $n\ge5$ be odd and, after applying an element of $\Gamma$, assume that
$\det X=-1$ and $\det Y=\det Z=1$. If the solution is not pairwise
commuting, then $X^2=\Id$ and
$\ord(Y),\ord(Z)\in\{3,4,6\}$.
\end{proposition}

\begin{proof}
Again $[A,B]\ne0$, and
$\Delta=-a^2+ab+b^2-5$. By Lemma~\ref{lem:mixed-delta-nonzero},
$\Delta\ne0$.

If $a\ne0$, then $|a|\ge11$. Choosing the two traces $u,v$ of
largest absolute value, exactly the same argument as in the previous
proposition gives
$\sigma_u^2\sigma_v^2>4|uv|$ and
$0<|\Delta|\le3|uv|+5<4|uv|$, a contradiction. Hence $a=0$ and
$c=b$, so $\Delta=b^2-5$.

Let $q=U_n(\tr Y,1)$ and $r=U_n(\tr Z,1)$. Since $[A,B]\ne0$ and
$B,C$ are nonscalar, $q,r\ne0$. The commutator-divisibility relations
give $q^2r^2\mid b^2-5$. If both $q$ and $r$ lie in the growth branch,
then $q^2r^2\ge(2|b|+3)^2>|b^2-5|$, impossible. If one coefficient lies
in the low-order branch, then $|b|\le1$. If the other lies in the growth
branch, then for $b=0$ its square is at least $4$ and cannot divide $5$;
for $b=\pm1$ its square is at least $9$ and cannot divide $4$. Thus
$|q|=|r|=1$ and $|b|\le1$. The trace-discriminant identity
$b^2-4=((\tr Y)^2-4)q^2$, together with the analogous formula for $Z$,
gives
$(\tr Y)^2=(\tr Z)^2=b^2\in\{0,1\}$. Hence each of $Y,Z$ has order
$3,4,$ or $6$. Finally, because $a=0$ and $\det A=-1$, the Cayley--Hamilton theorem
gives $A^2=\Id$. Thus $X^{2n}=\Id$, so $X$ has finite order. Since
$\det X=-1$, Lemma~\ref{lem:finite-orders} yields $X^2=\Id$.
\end{proof}

\begin{proposition}\label{prop:all-positive}
Let $n\ge5$ be odd. There is no noncommuting solution satisfying
$\det X=\det Y=\det Z=1$.
\end{proposition}

\begin{proof}
Suppose otherwise. By Lemma~\ref{lem:power-commuting-lift},
$[A,B]\ne0$. Since $A,B\in\SL_2(\mathbb Z)$,
Lemma~\ref{lem:SL2-singular-commutator} gives
$\Delta=\det[A,B]\ne0$. On the other hand, the commutator determinant
formula gives $\Delta=3-(a^2+ab+b^2)$. Let $u,v$ be the two traces of
largest absolute value.

If $|v|\ge3$, the two corresponding Cayley--Hamilton coefficients
both lie in the growth branch, so
$\sigma_u^2\sigma_v^2>4|uv|$. Moreover,
$a^2+ab+b^2=(a^2+b^2+c^2)/2\ge v^2\ge9$, and hence $\Delta\ne0$.
Lemma~\ref{lem:three-trace-bound} gives
$0<|\Delta|\le3|uv|+3<4|uv|$, contradicting the divisibility relation.
Thus $|v|\le2$ and consequently $|u|\le4$.

Let $d$ be the trace of any one of $A,B,C$, and let $\rho$ denote its
corresponding Cayley--Hamilton coefficient. The trace-discriminant
identity gives $\rho^2\mid d^2-4$. If $d=\pm3$, the corresponding power
matrix is nonscalar, so $\rho\ne0$. It cannot lie in the low-order branch
of Lemma~\ref{lem:growth-plus}, because that branch requires $|d|\le1$.
Thus it lies in the growth branch and $\rho^2\ge9$. But $\rho^2$ is a
positive square dividing $5$, so $\rho^2=1$, a contradiction. If
$d=\pm4$, the same argument gives $\rho^2\ge11$, whereas the only
positive squares dividing $12$ are $1$ and $4$, again a contradiction.
Therefore the trace of a positive-determinant $n$th power cannot be
$\pm3$ or $\pm4$, and
$a,b,c\in\{-2,-1,0,1,2\}$. Suppose that one of these traces is $\pm2$
and that the corresponding power is nonscalar. From
$d^2-4=((\tr M)^2-4)U_n(\tr M,1)^2$, the root trace must also be
$\pm2$. Hence the corresponding Cayley--Hamilton coefficient $p$
satisfies $|p|=n$ and $p^2\ge25$. The other two power matrices are also
nonscalar, so their Cayley--Hamilton coefficients are nonzero. If $\tau$
is either one of them, then $\tau^2\ge1$, and commutator divisibility
gives $p^2\tau^2\mid\Delta$. Since $\Delta\ne0$,
$|\Delta|\ge p^2\tau^2\ge25$. But $|a|,|b|\le2$, so
\[
\begin{aligned}
|\Delta|
&=|3-a^2-ab-b^2|\\
&\le3+a^2+|ab|+b^2\\
&\le15<25,
\end{aligned}
\]
a contradiction. If one of the powers is scalar, then $[A,B]=0$, also
contradicting the opening conclusion $[A,B]\ne0$.

It remains that all three traces lie in $\{-1,0,1\}$. If one coefficient
were still in the growth branch, the identity
$d^2-4=(t^2-4)u^2$ would give the following contradictions. For
$d=\pm1$, the square $u^2\ge5$ cannot divide $3$. For $d=0$, the
condition $u^2\ge3$ allows only the square $4$, but then
$t^2-4=-1$, so $t^2=3$, impossible. Hence all three coefficients lie in
the low-order branch. The proof of Lemma~\ref{lem:growth-plus} shows that
this branch is possible only when the root trace is $0$ or $\pm1$.
Consequently, $X,Y,Z$ are nonscalar positive-determinant torsion matrices,
and so are $A,B,C$. This contradicts
Theorem~\ref{thm:finite-orbits}, which has no all-positive orbit.
\end{proof}

\begin{theorem}\label{thm:odd-finite-order}
Let $n\ge5$ be odd.  If a solution is not pairwise commuting, then
$X,Y,Z$ are torsion matrices, and there is a unique
$\mathbf T\in\mathscr S_{\mathrm{ord}}$ such that
\[
(X^n,Y^n,Z^n)\in\mathcal O_G(\mathbf T).
\]
Equivalently, the ordered power triple belongs to the disjoint orbit
decomposition \eqref{eq:ordered-torsion-decomposition}; the auxiliary
$\Gamma$-normalization is no longer present in the conclusion.
\end{theorem}

\begin{proof}
Separate the proof into the four cases in which the number of
negative-determinant components is $0,1,2,$ or $3$.  The case of no
negative determinant is excluded by Proposition~\ref{prop:all-positive}.
The remaining three cases are reduced to torsion matrices by
Propositions~\ref{prop:one-negative-complete},
\ref{prop:two-negative-complete}, and
\ref{prop:three-negative-complete}, respectively.  The power triple is
therefore an element of $\mathcal N_{\mathrm{tor}}$, and the existence
and uniqueness of $\mathbf T$ follow from
Proposition~\ref{prop:ordered-torsion-orbits}.
\end{proof}

\subsubsection{The cubic equation}

\begin{lemma}
\label{lem:cubic-symmetry-list}
Put
\[
S_-:=\{0,\pm4,\pm14\},
\qquad
S_+:=\{0,\pm2,\pm18\}.
\]
For a determinant pattern
$(\varepsilon_X,\varepsilon_Y,\varepsilon_Z)$ chosen from
\[
(-1,-1,-1),\qquad(-1,-1,1),\qquad(-1,1,1),\qquad(1,1,1),
\]
let $a\in S_{\varepsilon_X}$, $b\in S_{\varepsilon_Y}$, and
$c\in S_{\varepsilon_Z}$ satisfy $c=a+b$, where
$S_{-1}=S_-$ and $S_1=S_+$.  The numbers of raw ordered trace triples
are respectively
\[
13,\qquad9,\qquad7,\qquad13,
\]
so there are $42$ raw triples in total.  Modulo the overall sign involution
and the stabilizer of the determinant pattern induced by $\Gamma$, these
split into respectively $3,4,4,3$ symmetry orbits, hence $14$ orbits in
total, with the following representatives:
\begin{align*}
(-1,-1,-1):\quad
 &(0,0,0),\ (-4,0,-4),\ (-14,0,-14),\\
(-1,-1,1):\quad
 &(0,0,0),\ (-4,4,0),\ (-14,14,0),\ (-14,-4,-18),\\
(-1,1,1):\quad
 &(0,0,0),\ (0,-2,-2),\ (0,-18,-18),\ (-4,2,-2),\\
(1,1,1):\quad
 &(0,0,0),\ (-2,0,-2),\ (-18,0,-18).
\end{align*}
\end{lemma}

\begin{proof}
For the all-negative pattern, write
$(\alpha,\beta,\gamma)=(a,b,-c)$.  Then
$\alpha+\beta+\gamma=0$ with all three entries in $S_-$.  The only
possibilities are $(0,0,0)$ and permutations of
$(s,-s,0)$ with $s\in\{4,14\}$.  For each value of $s$ there
are six permutations, so this gives $1+6+6=13$ raw triples before the
overall sign identification,
and the full induced $S_3$-action together with overall sign leaves the
three representatives displayed above.  The all-positive pattern is
identical with $S_+$ in place of $S_-$: the only possibilities are
$(0,0,0)$ and permutations of $(s,-s,0)$ with
$s\in\{2,18\}$; again there are $1+6+6=13$ raw triples and
three symmetry classes.

For $(-1,-1,1)$, the sums of two elements of $S_-$ that belong to
$S_+$ are only $0$ and $\pm18$.  If $c=0$, then $b=-a$, giving the five
raw triples with $a\in S_-$.  If $c=18$, the ordered pairs are
$(a,b)=(4,14),(14,4)$; if $c=-18$, they are their negatives.  Thus
there are $5+4=9$ raw triples.  Interchanging the first two variables
and applying the overall sign involution gives the four representatives
listed above.

Finally, for $(-1,1,1)$, the equation is $a=c-b$.  If $a=0$, then
$b=c\in S_+$, which gives five raw triples.  If $a=-4$, the only
possibility is $(b,c)=(2,-2)$, and if $a=4$ it is the negative triple;
no solution occurs for $a=\pm14$.  Hence there are $5+2=7$ raw triples.
The transformation
$\rho_{23}(X,Y,Z)=(X,-Z,-Y)$ identifies the two signs in the cases
$a=0$, while the overall sign identifies the two triples with
$|a|=4$.  This gives exactly the four representatives stated above.
\end{proof}

\begin{lemma}\label{lem:cubic-trace-check}
Let $X,Y,Z\in G$ satisfy $X^3+Y^3=Z^3$. Put
\[
\varepsilon_X=\det X,\qquad
\varepsilon_Y=\det Y,\qquad
\varepsilon_Z=\det Z,
\]
and let $x=\tr X$, $y=\tr Y$, and $z=\tr Z$. Define
\[
\begin{gathered}
a=x^3-3\varepsilon_Xx,\quad
b=y^3-3\varepsilon_Yy,\quad
c=z^3-3\varepsilon_Zz,\\[-1pt]
p=x^2-\varepsilon_X,\quad
q=y^2-\varepsilon_Y,\quad
r=z^2-\varepsilon_Z.
\end{gathered}
\]
If $pqr\ne0$ and $\Delta:=\det[X^3,Y^3]\ne0$, then $x=y=z=0$.
\end{lemma}

\begin{proof}
The trace relation gives $c=a+b$, and
Lemma~\ref{lem:comm-divisibility} gives
\[
p^2q^2\mid\Delta,\qquad
p^2r^2\mid\Delta,\qquad
q^2r^2\mid\Delta.
\]
Let $u,v$ be the two elements of $a,b,c$ of largest absolute value, and
write $\sigma_u,\sigma_v$ for the corresponding two coefficients among
$p,q,r$.
Suppose that $|v|\ge18$. Then the corresponding root traces have absolute
value at least $3$, and the cubic formulas give
$\sigma_u^2\ge\frac{25}{9}|u|$ and
$\sigma_v^2\ge\frac{25}{9}|v|$. We verify the first estimate; the second
is identical. If the corresponding root matrix has determinant $1$ and
root-trace absolute value $t\ge3$, then
$|u|=t(t^2-3)$ and $\sigma_u^2=(t^2-1)^2$, while
\[
\begin{aligned}
9(t^2-1)^2-25t(t^2-3)
={}&9(t-3)^4+83(t-3)^3\\
&+243(t-3)^2+264(t-3)+126>0.
\end{aligned}
\]
If the corresponding root matrix has determinant $-1$, then
$|u|=t(t^2+3)$ and $\sigma_u^2=(t^2+1)^2$, and
\[
9(t^2+1)^2-25t(t^2+3)
=(t-3)(9t^3+2t^2+24t-3)\ge0.
\]
Thus $\sigma_u^2\ge\frac{25}{9}|u|$ in both determinant cases, and
likewise for $v$. Hence
$\sigma_u^2\sigma_v^2\ge\frac{625}{81}|uv|$. On the other hand, for
each determinant-sign pattern, $\Delta$ is one of the quadratic forms in
Lemma~\ref{lem:three-trace-bound}, plus a constant $\pm3$ or $\pm5$.
Therefore
\[
|\Delta|\le3|uv|+5<\frac{625}{81}|uv|,
\]
because $|uv|\ge18^2$. This contradicts the relevant pairwise
divisibility relation.

Hence $|v|<18$. The nonzero absolute values of cubic traces are
$2,18,52,\ldots$ in the positive-determinant case and
$4,14,36,\ldots$ in the negative-determinant case. Therefore
$|v|\le14$ and $|u|\le28$, so
\[
\begin{aligned}
|x|,|y|,|z|&\le3
&&\text{on positive-determinant components},\\
|x|,|y|,|z|&\le2
&&\text{on negative-determinant components}.
\end{aligned}
\]
Since $pqr\ne0$, the values $x=\pm1$ are excluded on a
positive-determinant component. Thus every remaining component belongs
to one of the following two five-element lists:
\begin{equation}\label{eq:cubic-small-lists}
\begin{aligned}
\varepsilon=-1:&\quad
(x,a,|p|)\in
\{(0,0,1),(\pm1,\pm4,2),(\pm2,\pm14,5)\},\\
\varepsilon=1:&\quad
(x,a,|p|)\in
\{(0,0,1),(\pm2,\pm2,3),(\pm3,\pm18,8)\}.
\end{aligned}
\end{equation}
Here the signs within each triple are linked.

We now use only the symmetries already defined in
Section~\ref{sec:introduction}. The system consisting of $c=a+b$ and the
three pairwise divisibility conditions is invariant under the induced
$\Gamma$-action: determinant signs are permuted, the overall sign
involution sends $(a,b,c)$ to $(-a,-b,-c)$, and the three divisibility
conditions are merely permuted. Hence it is enough to consider the four
determinant patterns
\[
(-1,-1,-1),\qquad (-1,-1,1),\qquad (-1,1,1),\qquad (1,1,1).
\]
Within these four patterns we may also use the stabilizer of the pattern
in $\Gamma$: for the first and fourth patterns the whole induced
$S_3$-action is available; for $(-1,-1,1)$ we may interchange the first
two variables; and for $(-1,1,1)$ we may use
$\rho_{23}(X,Y,Z)=(X,-Z,-Y)$.

For $A=X^3$, $B=Y^3$, and $C=A+B=Z^3$, the identity
\[
\det(A+B)=\det A+\det B+\tr A\tr B-\tr(AB)
\]
gives
\[
\tr(AB)=\varepsilon_X+\varepsilon_Y+ab-\varepsilon_Z.
\]
Substitution into Lemma~\ref{lem:comm-det-formula} yields, for the four
canonical sign patterns,
\[
\begin{array}{c|c}
(\varepsilon_X,\varepsilon_Y,\varepsilon_Z)&\Delta\\
\hline
(-1,-1,-1)&a^2+ab+b^2+3\\
(-1,-1,1)&a^2+3ab+b^2-5\\
(-1,1,1)&-a^2+ab+b^2-5\\
(1,1,1)&3-a^2-ab-b^2.
\end{array}
\]

Lemma~\ref{lem:cubic-symmetry-list} proves, without any exhaustive
machine enumeration, that the remaining raw trace triples number
$13,9,7,13$ in the four determinant patterns and reduce to exactly
$3,4,4,3$ symmetry classes.  For those fourteen representatives, the
linked root traces in \eqref{eq:cubic-small-lists} determine
$p^2q^2$, and the preceding formulas determine $\Delta$.  The resulting
hand-checkable table is
\[
\begin{array}{c|c|r|r}
(\varepsilon_X,\varepsilon_Y,\varepsilon_Z)
&(a,b,c)&p^2q^2&\Delta\\
\hline
(-1,-1,-1)&(-14,0,-14)&25&199\\
           &(-4,0,-4)&4&19\\
           &(0,0,0)&1&3\\[1pt]
(-1,-1,1)&(-14,-4,-18)&100&375\\
          &(-14,14,0)&625&-201\\
          &(-4,4,0)&16&-21\\
          &(0,0,0)&1&-5\\[1pt]
(-1,1,1)&(-4,2,-2)&36&-25\\
         &(0,-18,-18)&64&319\\
         &(0,-2,-2)&9&-1\\
         &(0,0,0)&1&-5\\[1pt]
(1,1,1)&(-18,0,-18)&64&-321\\
       &(-2,0,-2)&9&-1\\
       &(0,0,0)&1&3
\end{array}
\]

Every nonzero row in the table violates the necessary condition
$p^2q^2\mid\Delta$. Thus only $(a,b,c)=(0,0,0)$ remains. By
\eqref{eq:cubic-small-lists}, the value $a=0$ (and similarly $b=0$ or
$c=0$) with the corresponding coefficient nonzero occurs only for root
trace $0$. Hence $x=y=z=0$.
\end{proof}

\begin{theorem}\label{thm:cubic-classification}
Every solution of $X^3+Y^3=Z^3$ satisfies
$\tr X=\tr Y=\tr Z=0$ and is a noncommuting torsion solution.  Moreover,
\begin{equation}\label{eq:cubic-orbit-decomposition}
\mathcal F_3
=
\mathcal F_3^{\mathrm{nc}}
=
\mathop{\bigsqcup}\limits_{\mathbf X\in
\mathscr R_3^{\mathrm{nc}}}
\mathcal O_G(\mathbf X),
\end{equation}
where $|\mathscr R_3^{\mathrm{nc}}|=14$.  Under the coarser equivalence
generated by $\Gamma$ and simultaneous conjugation, these fourteen
orbits form the three classes represented by
$\operatorname{rt}_3(\mathbf A)$,
$\operatorname{rt}_3(\mathbf B)$, and
$\operatorname{rt}_3(\mathbf C_4)$.
\end{theorem}

\begin{proof}
Put
$\varepsilon_X=\det X$, $\varepsilon_Y=\det Y$,
$\varepsilon_Z=\det Z$, and
$x=\tr X$, $y=\tr Y$, $z=\tr Z$. By the Cayley--Hamilton theorem,
$X^3=(x^2-\varepsilon_X)X-\varepsilon_Xx\Id$, with analogous formulas
for $Y,Z$.

We first show that none of $X^3,Y^3,Z^3$ is scalar. Suppose one of
them is scalar. After applying an element of $\Gamma$, assume that $X^3=k\Id$ with
$k\in\mathbb Z$. Since $X\in\GL_2(\mathbb Z)$,
$1=|\det(X^3)|=k^2$,
so $k=\pm1$. Write $X^3=\sigma\Id$ with
$\sigma\in\{\pm1\}$. From $Z^3=Y^3+\sigma\Id$ and the determinant
formula we obtain
\[
\tr(Y^3)=\sigma\bigl(\det(Z^3)-\det(Y^3)-1\bigr).
\]
The right-hand side is odd. But the trace $t^3-3\varepsilon t$ of the
cube of any integral $2\times2$ matrix is always even, a contradiction.
Thus all three cubes are nonscalar.

If $p=x^2-\varepsilon_X=0$, the Cayley--Hamilton expansion gives
$X^3=\mp\Id$, contradicting the preceding paragraph. Thus $p,q,r$ are
all nonzero.

Suppose that $[X^3,Y^3]=0$, and set $A=X^3$, $B=Y^3$, and
$C=Z^3=A+B$. Then $A,B,C$ commute pairwise and, as proved above, are all
nonscalar. By Lemma~\ref{lem:nonscalar-centralizer},
$C_{M_2(\mathbb Q)}(A)=\mathbb Q[A]$. Thus $B,C\in\mathbb Q[A]$. Since
$B,C$ are also nonscalar, Lemma~\ref{lem:nonscalar-centralizer} gives
$\mathbb Q[B]=\mathbb Q[A]=\mathbb Q[C]$.
On the other hand, $X$ commutes with $A=X^3$, $Y$ with $B=Y^3$, and
$Z$ with $C=Z^3$, so
\[
X\in\mathbb Q[A],\qquad
Y\in\mathbb Q[B]=\mathbb Q[A],\qquad
Z\in\mathbb Q[C]=\mathbb Q[A].
\]
Hence $X,Y,Z$ all lie in the commutative algebra $\mathbb Q[A]$ and
commute pairwise. This contradicts
Theorem~\ref{thm:commuting-solutions}, which gives no commuting solution
when $3\mid n$. Therefore $[X^3,Y^3]\ne0$.

We next show that $\Delta=\det[X^3,Y^3]\ne0$. If all three
determinants are $-1$, then $\Delta=a^2+ab+b^2+3>0$. For a mixed
determinant-sign pattern, apply Lemma~\ref{lem:mixed-delta-nonzero} after
applying an element of $\Gamma$. This does not change the conclusion for
the original triple: under every element of $\Gamma$, the
commutator of the new first two components is one of
$\pm[A,B]$, $\pm[A,C]$, or $\pm[B,C]$, and these matrices all have the
same determinant. Finally, suppose all three determinants are $1$. Then $A=X^3$ and $B=Y^3$ belong to $\SL_2(\mathbb Z)$. We have
already proved that $[A,B]\ne0$, so the contrapositive of
Lemma~\ref{lem:SL2-singular-commutator} gives
$\Delta=\det[A,B]\ne0$. Thus $\Delta\ne0$ for all eight determinant-sign
patterns.

All three commutator-divisibility relations are now available.
Lemma~\ref{lem:cubic-trace-check} gives $x=y=z=0$.  Hence
$X^2=-\det(X)\Id$, $Y^2=-\det(Y)\Id$, and
$Z^2=-\det(Z)\Id$, so all three matrices have finite order.
Proposition~\ref{prop:ordered-torsion-orbits} places the ordered power
triple in a unique orbit represented by
$\mathscr S_{\mathrm{ord}}$.  A cube cannot be a nonscalar matrix of
order $3$ or $6$, so the six representatives arising from the
$\mathbf C_3$-class are excluded.  The remaining $2+6+6=14$ ordered
power orbits contain only components of order $2$ or $4$.
Lemma~\ref{lem:finite-root-recovery} gives a unique cube root for each
component; equivalently, the componentwise cube-root map carries these
fourteen power-orbit representatives bijectively to
$\mathscr R_3^{\mathrm{nc}}$.  Taking cubes preserves simultaneous
conjugacy, so no two of the resulting root triples can become conjugate.
This proves the disjoint decomposition
\eqref{eq:cubic-orbit-decomposition}.
\end{proof}

\subsubsection{Root recovery and completion of the odd classification}

\begin{lemma}\label{lem:finite-root-recovery}
Let $A\in\GL_2(\mathbb Z)$ be a nonscalar torsion matrix of order
$d\in\{2,3,4,6\}$. Let $n$ be odd and suppose $\gcd(n,d)=1$. If
$T\in\GL_2(\mathbb Z)$ satisfies $T^n=A$, then
$T=A^e$, where $ne\equiv1\pmod d$.

In particular, when $d=2$ and $3\mid n$, the abstract order formula
appears to allow $\ord(T)=6$. However, every integral matrix of order $6$
satisfies $T^3=-\Id$, so its $n$th power is $-\Id$ and cannot equal the
nonscalar involution $A$. Thus in this case as well, $T=A$.
\end{lemma}

\begin{proof}
By the first part of Lemma~\ref{lem:torsion-root-centralizer}, $T$ has
finite order and commutes with $A$. Put $D=\ord(T)$. Then
$d=\ord(T^n)=D/\gcd(D,n)$, where $D\in\{1,2,3,4,6\}$.

If $3\nmid n$, Lemma~\ref{lem:torsion-root-centralizer} already gives
$T=A^e$ with $ne\equiv1\pmod d$. It remains only to consider $3\mid n$.
Since $\gcd(n,d)=1$, we must then have $d\in\{2,4\}$.

If $d=4$, the displayed order formula and
$D\in\{1,2,3,4,6\}$ force $D=4$. Thus $\gcd(D,n)=1$, and
$A=T^n$ and $T$ generate the same cyclic group of order $4$. Choosing
$ne\equiv1\pmod4$ yields $T=A^e$.

If $d=2$, formally $D=2$ or $D=6$. In the latter case,
Lemma~\ref{lem:finite-orders} gives $T^3=-\Id$. Since $n$ is odd and
$3\mid n$, we get $T^n=-\Id$, which cannot equal the nonscalar
involution $A$. Hence $D=2$, so $T^n=T=A$.
\end{proof}

Let \(\mathbf T=(A,B,C)\in\mathscr S_{\mathrm{ord}}\), and assume that
each component order is coprime to the odd integer \(n\), with the
order-two case interpreted as in Lemma~\ref{lem:finite-root-recovery}.
If the component orders are \(d_A,d_B,d_C\), let
\(e_A,e_B,e_C\) satisfy
\[
ne_A\equiv1\pmod{d_A},\qquad
ne_B\equiv1\pmod{d_B},\qquad
ne_C\equiv1\pmod{d_C}.
\]
Then all roots of the orbit \(\mathcal O_G(\mathbf T)\) are exactly
\begin{equation}\label{eq:root-recovery}
X=P A^{e_A}P^{-1},
\qquad
Y=P B^{e_B}P^{-1},
\qquad
Z=P C^{e_C}P^{-1},
\qquad P\in G.
\end{equation}
Equivalently, they form the single orbit
\(\mathcal O_G(\operatorname{rt}_n(\mathbf T))\).

\begin{lemma}
\label{lem:root-orbit-preservation}
Let \(\mathbf T,\mathbf T'\in\mathscr S_{\mathrm{ord}}\) be admissible
for the odd exponent \(n\).  Then
\[
\mathcal O_G(\operatorname{rt}_n(\mathbf T))
=
\mathcal O_G(\operatorname{rt}_n(\mathbf T'))
\]
if and only if
\[
\mathcal O_G(\mathbf T)=\mathcal O_G(\mathbf T').
\]
\end{lemma}

\begin{proof}
If the root triples are conjugate, taking \(n\)th powers shows that the
power triples are conjugate.  Conversely, suppose
\(\mathbf T'=P\mathbf T P^{-1}\) componentwise.  The conjugate
\(P\operatorname{rt}_n(\mathbf T)P^{-1}\) is an integral \(n\)th root
of \(\mathbf T'\).  Componentwise uniqueness in
Lemma~\ref{lem:finite-root-recovery} forces it to equal
\(\operatorname{rt}_n(\mathbf T')\).
\end{proof}

Thus the componentwise root map carries the ordered power transversal
\(\mathscr S_{\mathrm{ord}}\) bijectively to an ordered root transversal.
It is also \(\Gamma\)-equivariant on the admissible torsion triples:
for every \(\gamma\in\Gamma\),
\[
\operatorname{rt}_n(\gamma\mathbf T)
=
\gamma\operatorname{rt}_n(\mathbf T).
\]
Indeed, the signed variable operations commute with taking odd powers,
and componentwise uniqueness in
Lemma~\ref{lem:finite-root-recovery} then gives the displayed identity.
Consequently, root recovery preserves not only the ordered
simultaneous-conjugacy splitting but also the four coarse
\(\Gamma\)-classes.

If \(3\nmid n\), every member of \(\mathscr S_{\mathrm{ord}}\) is
admissible, and its image is precisely the set
\(\mathscr R_n^{\mathrm{nc}}\) in
\eqref{eq:canonical-odd-reps}.  If \(3\mid n\), the six members arising
from the \(\mathbf C_3\)-class contain components of order \(3\) or
\(6\) and cannot be \(n\)th powers; the remaining fourteen members are
admissible and map to \(\mathscr R_n^{\mathrm{nc}}\).

\begin{theorem}
\label{thm:odd-classification}
Let \(n\ge3\) be odd.
\begin{enumerate}[label=\textup{(\roman*)}]
\item If \(3\nmid n\), then
\[
\mathcal F_n^{\mathrm{com}}
=
\mathop{\bigsqcup}\limits_{s\in\mathbb Z/6\mathbb Z}
\mathcal O_G(Q^{s+e},Q^{s-e},Q^s),
\qquad ne\equiv1\pmod6,
\]
and
\[
\mathcal F_n^{\mathrm{nc}}
=
\mathop{\bigsqcup}\limits_{\mathbf X\in
\mathscr R_n^{\mathrm{nc}}}
\mathcal O_G(\mathbf X).
\]
Hence \(\mathcal F_n\) consists of six commuting and twenty
noncommuting simultaneous-conjugacy orbits.

\item If \(3\mid n\), then
\(\mathcal F_n^{\mathrm{com}}=\varnothing\) and
\[
\mathcal F_n
=
\mathcal F_n^{\mathrm{nc}}
=
\mathop{\bigsqcup}\limits_{\mathbf X\in
\mathscr R_n^{\mathrm{nc}}}
\mathcal O_G(\mathbf X),
\]
where \(|\mathscr R_n^{\mathrm{nc}}|=14\).
\end{enumerate}
\end{theorem}

\begin{proof}
The commuting assertions and the disjointness of their six orbits are
Theorem~\ref{thm:commuting-solutions}.  For \(n=3\), the noncommuting
statement is Theorem~\ref{thm:cubic-classification}.

Assume now that \(n\ge5\).  By
Theorem~\ref{thm:odd-finite-order}, the ordered power triple of every
noncommuting solution lies in a unique orbit represented by
\(\mathscr S_{\mathrm{ord}}\).  If \(3\nmid n\), all twenty
representatives are admissible.  If \(3\mid n\), exactly the six
representatives from the \(\mathbf C_3\)-class are inadmissible, since
an \(n\)th power of an integral torsion matrix cannot be a nonscalar
matrix of order \(3\) or \(6\); the remaining fourteen representatives
have component orders \(2\) or \(4\).  Lemma~\ref{lem:finite-root-recovery}
recovers every root component, and
Lemma~\ref{lem:root-orbit-preservation} transfers both completeness and
pairwise disjointness from the power orbits to the root orbits.  This
gives exactly the decompositions stated above.
\end{proof}

\subsection{Classification for admissible even exponents}
\label{sec:even-classification}

\begin{theorem}
\label{thm:even-classification}
Let \(n\ge3\) and \(n\equiv2,10\pmod{12}\).  Then
\begin{equation}\label{eq:even-equals-pythagorean}
\mathcal F_n=\mathcal P_+,
\end{equation}
and consequently
\begin{equation}\label{eq:even-orbit-decomposition}
\mathcal F_n
=
\mathop{\bigsqcup}\limits_{\epsilon_1,\epsilon_2\in\{\pm1\}}
\mathop{\bigsqcup}\limits_{T\in\mathcal T}
\mathcal O_G(-\epsilon_1R^2,-\epsilon_2R,T).
\end{equation}
Thus the admissible even Fermat equation has exactly the same ordered
solution set as the determinant-one matrix Pythagorean equation.
\end{theorem}

\begin{proof}
Write \(n=2m\).  Since \(n\ge3\) and
\(n\equiv2,10\pmod{12}\), we have \(m\ge5\), \(m\) odd, and
\(3\nmid m\).  Put
\[
A=X^n,\qquad B=Y^n,\qquad C=Z^n,
\qquad U=A^{-1}B.
\]
Then \(A,B,C\in\SL_2(\mathbb Z)\), \(A+B=C\),
\(\det U=1\), and \(\det(\Id+U)=1\).  Hence
\(\tr U=-1\) and \(U^2+U+\Id=0\).  After simultaneous integral
conjugation, we may assume \(U=R\).  Thus
\[
B=AR,
\qquad
C=A(\Id+R)=-AR^2.
\]

Set \(S=X^2\), \(T_0=Y^2\), and \(W=Z^2\).  Then
\[
S^m+T_0^m=W^m,
\qquad S,T_0,W\in\SL_2(\mathbb Z).
\]
By Proposition~\ref{prop:all-positive}, this odd-exponent equation is
commuting.  Hence \(A=S^m\) and \(B=T_0^m\) commute.  Since
\(B=AR\), we obtain
\[
0=[A,B]=[A,AR]=A[A,R],
\]
so \([A,R]=0\).  Therefore
\[
A\in C_G(R)=\langle Q\rangle
=\{\pm\Id,\pm R,\pm R^2\}.
\]

Substitution into \((A,B,C)=(A,AR,-AR^2)\) leaves, up to interchanging
the first two components, only
\begin{equation}\label{eq:even-power-triple}
(A,B,C)=(R,R^2,-\Id).
\end{equation}
To justify the power-admissibility test explicitly, note first that if a
torsion matrix $D$ satisfies $D=M^n$, then $M$ is itself torsion: if
$\ord(D)=d$, then $M^{nd}=\Id$.  Lemma~\ref{lem:finite-orders} therefore
reduces the test to root orders $1,2,3,4,6$.  Since
$n\equiv2\pmod4$ and $n\equiv2$ or $4\pmod6$, the only possible torsion
values of an $n$th power are $\Id$, $-\Id$, and nonscalar elements of
order $3$; the value $-\Id$ arises from a root of order $4$, whereas a
root of order $6$ has an $n$th power of order $3$.  In particular, no
nonscalar involution and no element of order $6$ can occur as an
$n$th power.  The remaining five values of $A$ force one component of
$(A,AR,-AR^2)$ to be inadmissible, exactly as recorded in the following
table:
\[
\begin{array}{c|c|c|c}
A&B&C&\text{Conclusion}\\
\hline
\Id&R&-R^2&-R^2\text{ is not an }n\text{th power}\\
-\Id&-R&R^2&-R\text{ is not an }n\text{th power}\\
R&R^2&-\Id&\text{admissible}\\
-R&-R^2&\Id&-R,-R^2\text{ are not admissible}\\
R^2&\Id&-R&-R\text{ is not admissible}\\
-R^2&-\Id&R&-R^2\text{ is not admissible}.
\end{array}
\]

From \eqref{eq:even-power-triple}, the roots \(X\) and \(Y\) lie in the
Eisenstein centralizer and \(Z\) has order \(4\).  More explicitly,
after restoring simultaneous conjugation,
\[
X=P Q^uP^{-1},
\qquad
Y=P Q^vP^{-1},
\qquad
Z=KJK^{-1},
\]
where
\[
nu\equiv2\pmod6,
\qquad
nv\equiv4\pmod6.
\]
The matrices \(P\) and \(K\) need not agree.  What matters for the
first two components is that the same matrix $P$ conjugates both roots.
Indeed, if $n\equiv2\pmod6$, then the congruences for $u$ and $v$ give
\[
X^2=PRP^{-1},\qquad Y^2=PR^2P^{-1};
\]
if $n\equiv4\pmod6$, the two right-hand sides are interchanged.
Consequently,
\[
X^2+Y^2=P(R+R^2)P^{-1}=-\Id.
\]
Since $Z$ has order $4$, Lemma~\ref{lem:finite-orders} also gives
$Z^2=-\Id$.  Hence $X^2+Y^2=Z^2$.
Moreover, \(X,Y,Z\) all have determinant \(1\).  Hence
\(\mathcal F_n\subseteq\mathcal P_+\).

Conversely, take a representative
\((-\epsilon_1R^2,-\epsilon_2R,T)\) in
\eqref{eq:Pplus-decomposition}.  Since
\(-R^2=Q\), \(-R=Q^5\), and \(n\) is even, the two signs disappear
on taking \(n\)th powers.  If \(n\equiv2\pmod6\), the first two
powers are \(R,R^2\); if \(n\equiv4\pmod6\), they are \(R^2,R\).
Thus their sum is always \(-\Id\).  Since \(T^2=-\Id\) and
\(n/2\) is odd,
\[
T^n=(T^2)^{n/2}=-\Id.
\]
Therefore every orbit in \(\mathcal P_+\) consists of solutions of the
\(n\)th-power equation, proving
\(\mathcal P_+\subseteq\mathcal F_n\).  This proves
\eqref{eq:even-equals-pythagorean}; the disjoint decomposition
\eqref{eq:even-orbit-decomposition} now follows from
Proposition~\ref{prop:pythagorean-orbit-input}.
\end{proof}

\subsection{Completion of the proof}

\begin{proof}[Proof of Theorem~\ref{thm:full-classification}]
Proposition~\ref{prop:no-4-divides} excludes every exponent divisible by
$4$: after passing to fourth powers, Proposition~\ref{prop:pythagorean-orbit-input}
forces a Gaussian component with square $-\Id$, and an $n$th root would
therefore produce an element of order $8$ in $G$, which is impossible.
Proposition~\ref{prop:no-6-divides} excludes every exponent divisible by
$6$.  Conversely, Proposition~\ref{prop:existence-construction} gives an
explicit solution whenever $4\nmid n$ and $6\nmid n$.  This proves
part~\textup{(i)}.

For even $n$, the two divisibility conditions are equivalent to
$n\equiv2$ or $10\pmod{12}$.  Theorem~\ref{thm:even-classification}
identifies the entire ordered solution set with the determinant-one
Pythagorean solution set and then applies the canonical transversal
$\mathcal T$.  This gives the disjoint decomposition in
part~\textup{(ii)}.

Now suppose that $n$ is odd.  Theorem~\ref{thm:odd-classification}
gives the complete ordered commuting and noncommuting decompositions for
every odd exponent $n\ge3$, and therefore proves
parts~\textup{(iii)} and~\textup{(iv)}.  In that theorem, the cubic case
$n=3$ is supplied by Theorem~\ref{thm:cubic-classification}, whereas
for $n\ge5$ the noncommuting classification follows from
Theorem~\ref{thm:odd-finite-order}, the ordered torsion-orbit
decomposition, and root recovery.
\end{proof}

\section{Conclusion}\label{sec:conclusion}

We have determined both the solvable exponents and the ordered
simultaneous-conjugacy orbits of the Fermat-type matrix equation
\eqref{eq:Fermat} over $\GL_2(\mathbb Z)$.  The argument keeps three
operations logically separate: the ordering of the variables, the
signed permutation symmetries of the odd equation, and simultaneous
integral conjugation.  This separation is reflected in the notation:
ordinary unions are used during symmetry reduction, whereas the final
classifications are genuine disjoint unions of $G$-orbits.

The equation is solvable exactly when $4\nmid n$ and $6\nmid n$.  Compared
with the determinant-one criterion of Qin \cite{Qin1996}, this shows
precisely where the full group contributes new solvable exponents: odd
multiples of $3$.  For an admissible even exponent, its full ordered
solution set is not merely
parametrized by independent Eisenstein and Gaussian conjugators; it is
exactly the determinant-one Pythagorean solution set
$\mathcal P_+$.  Consequently, the residual-orbit criterion and canonical transversal
$\mathcal T$ from the matrix Pythagorean equation give a nonredundant
disjoint orbit decomposition of every even solution; the precise companion
inputs are recorded in Remark~\ref{rem:companion-input}.

For odd exponents, the commuting and noncommuting parts have different
origins.  The commuting solutions are governed by the six units in the
Eisenstein order and form six disjoint $G$-orbits when $3\nmid n$; no
commuting solution exists when $3\mid n$.  Every noncommuting solution
has a torsion power triple.  After the coarse $\Gamma$-reduction to four
signed-permutation types, a further ordered orbit-splitting argument
gives twenty simultaneous-conjugacy orbits when $3\nmid n$.  If
$3\mid n$, the six orbits in the $\mathbf C_3$ family have no admissible
$n$th roots, leaving exactly fourteen.  Thus the corresponding coarse
counts are four and three, but the ordered orbit counts are twenty and
fourteen.  The cubic exponent is handled in the main text by the elementary
symmetry-reduced trace argument of Lemma~\ref{lem:cubic-trace-check},
while the larger odd exponents follow from the Cayley--Hamilton,
commutator-divisibility, and trace-growth reductions.

\section*{Funding}
This work was supported by the Project of Guangdong University of Foreign
Studies (Grant No.~2024RC063).

\end{document}